\documentclass[12pt]{amsart}
\usepackage{etex}

\usepackage[margin=1in]{geometry}

\usepackage{amscd,latexsym,amsthm,amsfonts,amssymb,amsmath,amsxtra,mathdots,tikz}
\usepackage[mathscr]{eucal}
\usepackage{xcolor}
\usepackage[normalem]{ulem}
\usepackage{hyperref}
\usepackage[all,cmtip]{xy}
\usepackage{enumitem}
\usepackage{tabularx}
\usepackage{textcomp}
\usepackage{caption}

\renewcommand\theequation{\thesection.\arabic{equation}}

\newcommand{\BA}{{\mathbb {A}}}

\newcommand{\BC}{{\mathbb {C}}}

\newcommand{\BZ}{{\mathbb {Z}}}

\newcommand{\CA}{{\mathcal {A}}}

\newcommand{\CE}{{\mathcal {E}}}

\newcommand{\CJ}{{\mathcal {J}}}

\newcommand{\CO}{{\mathcal {O}}}
\newcommand{\CP}{{\mathcal {P}}}

\newcommand{\CS}{{\mathcal {S}}}

\newcommand{\Fg}{{\mathfrak {g}}}

\newcommand{\Fl}{{\mathfrak {l}}}
\newcommand{\Fm}{{\mathfrak {m}}}
\newcommand{\Fn}{{\mathfrak {n}}}

\newcommand{\Fu}{{\mathfrak {u}}}

\newcommand{\GL}{{\mathrm{GL}}}

\newcommand{\GSp}{{\mathrm{GSp}}}

\newcommand{\GSO}{{\mathrm{GSO}}}
\newcommand{\GSpin}{{\mathrm{GSpin}}}
\newcommand{\Spin}{{\mathrm{Spin}}}
\newcommand{\PGSO}{{\mathrm{PGSO}}}

\newcommand{\PGL}{{\mathrm{PGL}}}

\newcommand{\SL}{{\mathrm{SL}}}

\newcommand{\HSpin}{{\mathrm{HSpin}}}
\newcommand{\PGSp}{{\mathrm{PGSp}}}

\newcommand{\SO}{{\mathrm{SO}}}

\newcommand{\Sp}{{\mathrm{Sp}}}

\newcommand{\tr}{{\mathrm{tr}}}

\newcommand{\vol}{{\mathrm{vol}}}

\newcommand{\wt}{\widetilde}

\newcommand{\pair}[1]{\langle {#1} \rangle}

\newcommand{\back}{\backslash}

\def\A{{\mathbb A}}
\def\BA{{\mathbb A}}

\def\BC{{\mathbb C}}

\def\Fu{{\mathfrak u}}
\def\back{{\backslash}}

\newtheorem{thm}{Theorem}[section]

\newtheorem{prop}[thm]{Proposition}
\newtheorem {conj}[thm]{Conjecture}
\newtheorem {assu}[thm]{Assumption}
\newtheorem {ques/conj}[thm]{Question/Conjecture}

\newtheorem{defn}[thm]{Definition}
\newtheorem{rmk}[thm]{Remark}

\makeatletter

\newcommand{\Rmnum}[1]{\expandafter\@slowromancap\romannumeral #1@}
\makeatother

\begin{document}
\renewcommand{\theequation}{\arabic{equation}}
\numberwithin{equation}{section}

\title[BZSV duality]{BZSV duality for some strongly tempered spherical varieties}

\author{Zhengyu Mao}
\address{Department of Mathematics \& Computer Science\\
Rutgers University – Newark\\
Newark, NJ 07102, USA}
\email{zmao@rutgers.edu}

\author{Chen Wan}
\address{Department of Mathematics \& Computer Science\\
Rutgers University – Newark\\
Newark, NJ 07102, USA}
\email{chen.wan@rutgers.edu}

\author{Lei Zhang}
\address{Department of Mathematics\\
National University of Singapore, Singapore}
\email{matzhlei@nus.edu.sg}

\date{}

\subjclass[2020]{Primary 11F67; 11F72}

\keywords{relative Langlands duality, strongly tempered spherical varieties}

\begin{abstract}
We propose two families of relative trace formula comparisons in the study of relative Langlands duality conjectured by Ben-Zvi--Sakellaridis--Venkatesh. This allows us to incorporate numerous relative trace formula comparisons studied during the last four decades under the BZSV duality framework.  For the proposed relative trace formula comparisons associated to some strongly tempered spherical varieties, we will prove the fundamental lemma and smooth transfer in the $p$-adic case. Moreover, inspired by the BZSV duality conjecture, we propose a conjecture regarding the degenerate Whittaker period, which generalizes Lapid-Mao's conjecture of the Whittaker period. 
\end{abstract}

\maketitle

\tableofcontents

\section{Introduction}

\subsection{The relative Langlands duality conjecture}
In \cite{BSV}, Ben-Zvi, Sakellaridis, and Venkatesh proposed a beautiful relative Langlands duality for spherical varieties (in this paper, we will call it BZSV duality). We will briefly explain their conjecture in this subsection. Throughout this paper, $k$ is a global field, $\BA=\BA_k$, $F$ is a local field, and $\psi$ is a non-trivial additive character of $\BA/k$ (resp. $F$) if we are in the global (resp. local) setting. 

The BZSV duality concerns a pair of dual data $(\Delta,\check{\Delta})$ where each side contains 4 datum: $\Delta=(G,H,\rho_H,\iota)$ and $\hat{\Delta}=(\hat{G},\hat{H}',\rho_{\hat{H}'},\hat{\iota}')$. Here $G$ is a split reductive group;  $H$ is a split reductive subgroup of $G$;  $\rho_H$  is a symplectic representation of $H$; and  $\iota$ is a homomorphism from $\SL_2$ into $G$ whose image commutes with $H$. The map $\iota$ induces a homomorphism $H\times \SL_2\rightarrow G$, which will still be denoted by $\iota$.  This map induces an adjoint action of $H\times \SL_2$ on the Lie algebra $\Fg$  of $G$ and we can decompose it as
$$\oplus_{k\in I} \rho_k\otimes Sym^k$$
where $\rho_k$ is some representation of $H$ and $I$ is a finite subset of $\BZ_{\geq 0}$. We let $I_{odd}$ be the subset of $I$ containing all the odd numbers. There are two main requirements for the quadruple $(G,H,\rho_H,\iota)$.

\begin{enumerate}
\item The representation $\rho_{H,\iota}=\rho_H\oplus (\oplus_{i\in I_{odd}} \rho_i)$ is a symplectic anomaly-free representation (see Section 5 of \cite{BSV}) of $H$.
\item The Hamiltonian space associated to the quadruple $(G,H,\rho_H,\iota)$ (defined in Section 3 of \cite{BSV}) is hyperspherical (Section 3.5 of \cite{BSV}). In particular, its generic stabilizer is connected.
\end{enumerate}

We refer the reader to \cite{BSV} for more details. Note that under BZSV duality, the group  $\hat{G}$ is the Langlands dual group of $G$ and $\hat{H}'=\hat{G}_{\Delta}$  can be viewed as the ``dual group" of the quadruple $\Delta$ (note that the groups $H$ and $\hat{H}$' are not dual to each other in general, and the nilpotent orbits $\iota$ and $\hat{\iota}'$ are also not dual to each other in general).  Let $(G,H,\rho_H,\iota)$ and $(\hat{G},\hat{H}',\rho_{\hat{H}'},\hat{\iota}')$ be two quadruples that are dual to each other under the BZSV duality. We use $\rho_{H,\iota}$ and $\rho_{\hat{H}',\hat{\iota}'}$ to denote the symplectic anomaly-free representations associated to these quadruples. As we explained above, the maps $\iota$ and $\hat{\iota}'$ induce adjoint actions of $H\times \SL_2$ (resp. $\hat{H}'\times \SL_2$) on $\Fg$ (resp. $\hat{\Fg}$) and they can be decomposed as 
$$\Fg=\oplus_{k\in I} \rho_k\otimes Sym^k,\;\hat{\Fg}=\oplus_{k\in \hat{I}} \hat{\rho}_k\otimes Sym^k$$
where $\rho_k$ (resp. $\hat{\rho}_k$) are representations of $H$ (resp. $\hat{H}'$). It is clear that the adjoint representation of $H$ (resp. $\hat{H}'$) is a subrepresentation of $\rho_0$ (resp. $\hat{\rho}_0$). 

To simplify the notation, we assume that the center of $G$ (resp. $\hat{G}$) is finite \footnote{otherwise we need to modulo the center in the definition of the period integrals}. Now we define the period integrals associated to the datum. We have a symplectic representation $\rho_{H,\iota}:H\rightarrow \Sp(V)$. Let $Y$ be a maximal isotropic subspace of $V$ and $\Omega_\psi$  be the Weil representation of $\widetilde{\Sp}(V)$ on the Schwartz space $\CS(Y(\BA))$. The anomaly free condition on $\rho_{H,\iota}$ ensures $\widetilde{\Sp}(V)$ splits over $Im(\rho_{H,\iota})$ and $\Omega_\psi$ restricts to a representation of $H(\BA)$ on $\CS(Y(\BA))$. We define the theta series
$$\Theta_{\psi}^{\varphi}(h)=\sum_{X\in Y(k)} \Omega_{\psi}(h)\varphi(X),\;h\in H(\BA),\varphi\in \CS(Y(\BA)),$$
and we can define the period integral to be
$$\CP_{H,\iota,\rho_H}(\phi,\varphi)=\int_{H(k)\back H(\BA)} \CP_\iota(\phi)(h)\Theta_{\psi}^{\varphi}(h)dh.$$
Here $\CP_\iota$ is the degenerate Whittaker period associated to $\iota$ (we refer the reader to equation \eqref{eq: defdegWhit} or \eqref{eq: defdegWhit2} in \S\ref{sec: degWhit} for its definition). Similarly we can also define the period integral $\CP_{\hat{H}',\hat{\iota}',\rho_{\hat{H}'}}(\phi,\varphi)$ on the dual side for automorphic forms $\phi$ of $\hat{G}(\BA)$ and Schwartz functions $\varphi$ of $\hat{Y}'(\BA)$ where $\hat{Y}'$ is a maximal isotropic subspace of the symplectic representation $\rho_{\hat{H}',\hat{\iota}'}$. The following conjecture is the main conjecture regarding global periods in BZSV duality.

\begin{conj}\label{BSV conj} (Ben-Zvi--Sakellaridis--Venkatesh, \cite{BSV})
\begin{enumerate}
\item Let $\pi$ be an irreducible discrete automorphic representation of $G(\BA)$ and let $\nu:\pi\rightarrow L^2(G(k)\back G(\BA))_{\pi}$ be an embedding. Then the period integral 
$$\CP_{H,\iota,\rho_H}(\phi,\varphi),\;\phi\in Im(\nu),\;\varphi\in \CS(Y(\BA))$$
is nonzero only if the Arthur parameter of $\pi$ factors through $\hat{\iota}':\hat{H}'(\BC)\times \SL_2(\BC)\rightarrow \hat{G}(\BC)$. If this is the case, $\pi$ is a lifting of a global tempered Arthur packet $\Pi$ of $H'(\BA)$ (the Langlands dual group of $\hat{H}'$). Then we can choose the embedding $\nu$ so that 
$$\frac{|\CP_{H,\iota,\rho_H}(\phi,\varphi)|^2}{\pair{\phi,\phi}}``=" \frac{L(1/2,\Pi,\rho_{\check{H}})\cdot\Pi_{k\in \hat{I}}L(k/2+1,\Pi,\hat{\rho}_k)}{L(1,\Pi,Ad)^2},\;\phi\in Im(\nu),\;\varphi\in \CS(Y(\BA)).$$
Here $\pair{,}$ is the $L^2$-norm, and $``="$ means the equation holds up to some Dedekind zeta functions, some global constant determined by the component group of the global  L-packet associated to $\pi$, and some finite product over the ramified places  (including all the archimedean places).
\item Let $\pi$ be an irreducible discrete automorphic representation of $\hat{G}(\BA)$ and let $\nu:\pi\rightarrow L^2(\hat{G}(k)\back \hat{G}(\BA))_{\pi}$ be an embedding. Then the period integral 
$$\CP_{\hat{H}',\hat{\iota}',\rho_{\hat{H}'}}(\phi,\varphi),\;\phi\in Im(\nu),\;\varphi\in \CS(\hat{Y}'(\BA))$$ 
is nonzero only if the Arthur parameter of $\pi$ factors through $\iota:H(\BC)\times \SL_2(\BC)\rightarrow G(\BC)$. If this is the case,  $\pi$ is a lifting of a global tempered Arthur packet $\Pi$ of $\hat{H}(\BA)$ (the Langlands dual of $H$). Then we can choose the embedding $\nu$ so that 
$$\frac{|\CP_{\hat{H}',\hat{\iota}',\rho_{\hat{H}'}}(\phi,\varphi)|^2}{\pair{\phi,\phi}}``=" \frac{L(1/2,\Pi,\rho_H)\cdot\Pi_{k\in I}L(k/2+1,\Pi,\rho_k)}{L(1,\Pi,Ad)^2},\;\phi\in Im(\nu),\;\varphi\in \CS(\hat{Y}'(\BA)).$$
\end{enumerate}
\end{conj}

\begin{rmk}\label{rmk BSV conj}
The above conjecture is usually called the Ichino-Ikeda type conjecture. To state an explicit identity instead of $``="$, one needs to make two adjustments on the right hand side of the equation.
\begin{itemize}
\item In the ramified places, instead of using the local L-function, one needs to use the so-called local relative character defined by the (conjectural) Plancherel decomposition (see Section 17 of \cite{SV} and Section 9 of \cite{BSV}).
\item One also needs to add some Dedekind zeta functions on the right hand side determined by the groups $G$ and $H$ (in all the known examples, those zeta functions are the $L$-function of the dual $M^{\vee}$ to the motive $M$ associated to $G,H$ introduced by Gross in \cite{Gross}), as well as some global constant determined by component group of the global  L-packet associated to $\pi$ (see Section 14.6.4 of \cite{BSV}).
\end{itemize}

\end{rmk}

The main goal of this paper is to propose a relative trace formula approach to establish many cases of the above conjecture. In \S\ref{sec RTF} and \S\ref{sec RTF2} we propose two families of relative trace formula comparisons. We explain how the relative trace formula comparisons together with a conjecture on the degenerate Whittaker period described in \S\ref{sec: degWhit} imply the equation on the period in Conjecture~\ref{BSV conj}. Some examples of our proposed relative trace formula comparisons are already established (or conjectured) in \cite{FJ, JR, JMR, MR, MR4, MR3, Z}.
We therefore incorporate all these cases under the same framework of BZSV duality. We will provide further evidence of Conjecture~\ref{BSV conj} by establishing the fundamental lemma and smooth transfers for the relative trace formula comparisons in all cases of strongly tempered spherical variety whose dual side is of rank one.

 Namely, we consider the following 6 strongly tempered spherical varieties from the previous paper \cite{WZ} of the second and third authors (we refer the reader to \cite{WZ} and Section 2 for the definition of the unipotent group $U$ of each case). For each of the models in the table below, one can associate dual data $(\Delta,\hat{\Delta})$ in BZSV duality as follows: one quadruple $\Delta$ is $(G,H_0,0,\iota)$ where $\iota$ is the $\SL_2$-embedding associated to the Richardson orbit of $U$ which will be explicitly defined in section \S\ref{sec DMVexplicit} and $0$ is just the trivial 0-dimensional symplectic representation; the dual quadruple $\hat{\Delta}$ is $(\hat{G},\hat{G},\hat{\rho},1)$ where $1$ denotes the $\SL_2$-embedding that maps all the elements to the identity.

\begin{figure}[h!]
\begin{tabular}{| c | c | c |c|c| c|}
\hline
\textnumero & $G$ & $H=H_0\ltimes U$ & $\hat{G}$ &  $\hat{\rho}$  \\
\hline
1 &  $\PGL_6$ & $\PGL_2\ltimes U$ & $\SL_6$ & $\wedge^3$   \\
\hline
2 & $\GSO_8\times \GL_2/\GL_1$ & $\PGL_2\ltimes U$ & $S(\GSpin_8\times \GL_2)$ & $\HSpin_8\otimes {\rm std}_{2}$   \\
\hline
3 & $\PGSO_{12}$ & $\PGL_2\ltimes U$  & $\Spin_{12}$ & $\HSpin_{12}$   \\
\hline
4 & $E_{7,ad}$ & $\PGL_2\ltimes U$ & $E_{7,sc}$  &$\omega_7$  \\
\hline
5 & $\PGSp_{10}$ & $\PGL_2\ltimes U$ & $\Spin_{11}$ & $\Spin_{11}$   \\
\hline
6 & $\GSp_{6}\times \GL_2/\GL_1$ & $\PGL_2\ltimes U$ & $S(\GSpin_7\times \GL_2)$ & $\Spin_{7}\otimes {\rm std}_2$   \\
\hline
\end{tabular}
\captionof{table}{}
\label{fig:1}
\end{figure}

Here $S(\GSpin_n\times \GL_2)=\{(g_1,g_2)\in \GSpin_n\times \GL_2|\;l(g_1)\det(g_2)=1\}$ where $l$ is the similitude character and $\omega_7$ is the 56-dimensional standard representation of $E_7$. These are the only models in \cite{WZ} where the reductive part  $H_0$  is of rank 1. In other words, the dual group of the quadruple $\hat{\Delta}$ is of rank 1 and hence is easy to study. 

%One of the main goals of the paper is to provide strong evidence for the above conjecture in the cases of these strongly tempered spherical varieties.

 For each of the models in the table, we will write down  Conjecture \ref{BSV conj} more explicitly. Indeed, a more explicit version of Part (1) Conjecture \ref{BSV conj} in these cases is stated in \cite[Conjecture 1.7]{WZ}. In this paper, we will concentrate on the dual side, propose a relative trace formula comparison for the dual side, and prove the fundamental lemma and smooth transfer, therefore providing strong evidence for part (2) of  Conjecture \ref{BSV conj} in the cases of these strongly tempered spherical varieties.
 % We refer the reader to Section 2-5 for more details.

\begin{rmk}
In \cite{BSV}, the authors considered only the cases when $G,H$ are split. When the groups $G,H$ are not split, it is not clear at this moment how to define the duality. One of the main obstacles is to define the L-group of quadruple in the non-split case. At this moment, to our knowledge, there are only two cases of duality considered in the non-split case. One is the Galois model case considered by Prasad in \cite{P}. The other one is for the model $(Res_{K/F}\GL_2,\GL_2)$ considered by the first author and Rallis in \cite{MR2} where $K/F$ is a cubic extension.
\end{rmk}

\subsection{A conjecture for degenerate Whittaker period}\label{sec: degWhit}
Motivated by the BZSV conjecture, we make a general conjecture about the degenerate Whittaker coefficient in this subsection. Let $\iota$ be a map from $\SL_2$ into a split reductive group $G$ and let $\CO_\iota$ be the nilpotent orbit of $\Fg$ associated to it. Let 
$$U=\{g\in G|\;\lim_{t\rightarrow 0}\iota(diag(t,t^{-1})) g\iota(diag(t,t^{-1}))^{-1}=1\}.$$
If the nilpotent orbit $\CO_\iota$ is even in the sense of Section 2.3 of \cite{GJ}, $\psi$ induces a generic character $\xi$ of $U(k)\back U(\BA)$ (see Section 2.3.1 of \cite{GJ}) and we define the degenerate Whittaker coefficient to be
\begin{equation}\label{eq: defdegWhit}
	\CP_\iota(\phi)=\int_{U(k)\back U(\BA)}\phi(u)\xi(u)du. 
	\end{equation}
If $\CO_{\iota}$ is not even, then $\psi$ induces a Weil representation $\Omega_{\psi}'$ of $U(\BA)$ on the space of Schwartz functions of some vector space $Z(\BA)$ (see Section 2.3.2 of \cite{GJ}). For $\varphi\in \CS(Z(\BA))$, we can define the theta series 
$$\Theta_{U,\psi}^{\varphi}(u)=\sum_{X\in Z(k)}\Omega_{\psi}'(u)\varphi(X)$$
and we define the degenerate Whittaker coefficient to be 
\begin{equation}\label{eq: defdegWhit2}
	\CP_\iota(\phi,\varphi)=\int_{U(k)\back U(\BA)}\phi(u)\Theta_{U,\psi}^{\varphi}(u)du.
	\end{equation}
To simplify the notation, we will skip $\varphi$ and still use $\CP_\iota(\phi)$ to denote the degenerate Whittaker coefficient. When $\CO_\iota$ is even (resp. non-even), $\CP_\iota$ is also called the Bessel period (resp. the Fourier-Jacobi period).

Given a nilpotent orbit $\hat{\CO}_{\iota}$ of $\hat{\Fg}$, by the duality introduced in the appendix of \cite{BV}, it induces a nilpotent orbit $\CO_\iota$ of $G$ that is special (Definition 1.10 of \cite{BV}). The nilpotent orbit $\hat{\CO}_\iota$ induces a homomorphism $\hat{\iota}:\SL_2\rightarrow \hat{G}$. Let $\hat{H}_{\hat{\iota}}$ be the neutral component of the centralizer of $Im(\hat{\iota})$ in $\hat{G}$. We get a homomorphism $\hat{H}_{\hat{\iota}}\times \SL_2\rightarrow \hat{G}$ which will still be denoted by $\hat{\iota}$. This map induces an adjoint action of $\hat{H}_{\hat{\iota}}\times \SL_2$ on $\hat{\Fg}$ and we can decompose it as
$$\oplus_{k\in \hat{I}} \hat{\rho}_k\otimes Sym^k$$
where $\hat{\rho}_k$ is some representation of $\hat{H}_{\hat{\iota}}$ and $\hat{I}$ is a finite subset of $\BZ_{\geq 0}$. 

Since the adjoint representation is of orthogonal type, we know that $\hat{\rho}_k$ is of symplectic type when $k$ is odd. We can decompose the symplectic representation $\oplus_{k\in \hat{I}_{odd}}\hat{\rho}_k$ (here $\hat{I}_{odd}$ is again the subset of odd integers) as 
\begin{equation}\label{decomposition of odd 1}
\oplus_{k\in \hat{I}_{odd}}\hat{\rho}_k=(\oplus_{j}\tau_j\oplus \tau_{j}^{\vee})\oplus (\oplus_{i} \sigma_i) 
\end{equation}
where $\sigma_i$ are distinct self-dual irreducible representations of $\hat{H}$ of symplectic type. In particular, $\sigma_i$ are those irreducible symplectic representations that appear odd times in $\oplus_{k\in \hat{I}_{odd}}\hat{\rho}_k$. We then define
\begin{equation}\label{decomposition of odd 2}
\hat\rho_{\hat{\iota}}=\oplus_{i} \sigma_i.
\end{equation}

Let $\pi$ be an irreducible discrete  automorphic representation of $G(\BA)$ and let $\nu:\pi\rightarrow L^2(G(k)\back G(\BA))_{\pi}$ be an embedding. We assume that the Arthur parameter of $\pi$ factors through $\hat{\iota}$ (i.e. $\pi$ is a lifting of a global tempered Arthur packet $\Pi$ of $H_{\hat\iota}(\BA)$ where $H_{\hat\iota}$ is the dual group of $\hat{H}_{\hat\iota}$). 

\begin{conj}\label{conj Whittaker general}
With the notation above, one can choose the embedding $\nu$ so that 
$$\frac{|\CP_\iota(\phi)|^2}{\pair{\phi,\phi}}``=" \frac{L(1/2,\Pi,\hat\rho_{\hat{\iota}})}{\Pi_{k\in \hat{I}}L(k/2+1,\Pi,\hat{\rho}_k)},\;\phi\in \nu(\pi).$$
%Here $\pair{,}$ is the $L^2$-norm and $\hat{I}_{odd}$ is the subset of $\hat{I}$ containing all the odd numbers.
\end{conj}

\begin{rmk}
As in Conjecture \ref{BSV conj}, here $``="$ means the equation holds up to some Dedekind zeta functions, some global constant determined by the component group of the global  L-packet associated to $\pi$ and some finite product over the ramified places  (including all the archimedean places). As we mentioned in Remark \ref{rmk BSV conj}, to state an explicit identity,  we need to replace the local L-function with the relative character at the ramified places, and we also need to add some Dedekind zeta functions as well as some global constant determined by the component group of the global  L-packet associated to $\pi$.

However, in the case if $\CO_\iota$ is not the regular nilpotent orbit, it is not clear to us how to define the local relative character at this moment. The reason is that when $\CO_\iota$ is not regular, the degenerate Whittaker model is not unique, and one can not define the local relative character using the Plancherel decomposition as in Section 17 of \cite{SV} and Section 9 of \cite{BSV}. In the regular case, the local relative character was defined by Lapid-Mao in \cite{LM}, which allows them to formulate the Ichino-Ikeda type conjecture for the Whittaker period.
\end{rmk}

When $\hat{\CO}_\iota$ is the zero nilpotent orbit, Conjecture \ref{conj Whittaker general} is compatible with the conjecture of Lapid-Mao for Whittaker period in \cite{LM}. 

When $G=\GL_{2n}$ (resp. $\SO_{2n}$, $\Sp_{4n}$), $\hat{Q}=\hat{L}\hat{N}$ is the parabolic subgroup of $\hat{G}$ with $L=\GL_n\times \GL_n$ (resp. $\GL_n$, $\GL_{2n}$) and $\hat{\CO}_\iota$ is the principal orbit of $\hat{L}$, the above conjecture is compatible with the unramified computation of the local intertwining operator in Proposition 4 of \cite{LR}.

In addition to the above examples, we have some other evidence for the above conjecture from the relative trace formula comparisons we propose below. If the relative trace formula comparisons hold, then Conjecture \ref{conj Whittaker general} is compatible with Conjecture~\ref{BSV conj}. We have much evidence for the proposed relative trace formulas through the prior works in \cite{FJ, JR, JMR, MR, MR4, MR3, Z}, as well as the comparisons proved in this paper.

 %In Section 4-6 we will provide more examples for this conjecture. 

In all examples relevant to the evidence mentioned above, the nilpotent orbits $\hat{\CO}_\iota$  are principal in some Levi subgroups of $\hat{G}$. On the other hand, we do have examples when $\hat{\CO}_\iota$ is not special, and we also have examples where the set $\{\tau_j\oplus \tau_{j}^{\vee}\}$ in the decomposition \eqref{decomposition of odd 1} is not empty. We will describe two such examples in \S\ref{sec orbitexample}.

%In the next two subsections, based on Conjecture \ref{BSV conj} and \ref{conj Whittaker general}, we will propose some relative trace formula comparisons for two categories of spherical varieties.

\subsection{A conjectural relative trace formula comparison for quadruples not related to central values}\label{sec RTF}

Let $(G,H,\rho_H,\iota)$ and $(\hat{G},\hat{H}',\rho_{\hat{H}'},\hat{\iota}')$ be two quadruples that are dual to each other under the BZSV duality. We use $\rho_{H,\iota}$ and $\rho_{\hat{H}',\hat{\iota}'}$ to denote the symplectic anomaly-free representations associated to these quadruples. As we explained above, the map $\hat{\iota}'$ induces adjoint actions of  $\hat{H}'\times \SL_2$ on $\hat{\Fg}$ and it can be decomposed as 
$$\hat{\Fg}=\oplus_{k\in \hat{I}} \hat{\rho}_k\otimes Sym^k$$
where $\hat{\rho}_k$ are representations of  $\hat{H}'$. We let $\CO_{\iota'}$ be the nilpotent orbit of $\Fg$  that is dual to $\hat{\iota}'$. In the previous two subsections we have defined the period integrals $\CP_{H,\iota,\rho_H}$ and $\CP_{\iota'}$. In this subsection we make the following assumption.

\begin{assu}\label{assumption 1}
The symplectic representations $\rho_{\hat{H}'}$ and $\hat{\rho}_{\hat{\iota}'}$ (define in \eqref{decomposition of odd 2}) are zero dimensional. 
\end{assu}

\begin{rmk}
According to Conjecture \ref{BSV conj} and \ref{conj Whittaker general}, the above assumption holds if the half integer values of automorphic $L-$functions do not appear in the numerator in the conjectured identity for the period integral $\CP_{H,\iota,\rho_H}$. However Assumption~\ref{assumption 1} may still hold even when the numerator involves the half integer values. For example when  $(G,H,\rho_H,\iota)=(\GL_n,\GL_{n-1},0,1)$ for $n>2$, by Table (3) of \cite{S} the half integer values of $L-$functions appear in the numerator when $n$ is even, meanwhile Assumption~\ref{assumption 1} holds (see \S\ref{sec orbitexample} for details).
\end{rmk}

Combining the above assumption with Conjecture \ref{BSV conj} and \ref{conj Whittaker general}, we expect the following conjecture.

\begin{conj}\label{conj rtf period 1}
Assume Assumption \ref{assumption 1} holds. Let $\pi$ be an irreducible discrete automorphic representation of $G(\BA)$ and let $\nu:\pi\rightarrow L^2(G(k)\back G(\BA))_{\pi}$ be an embedding. Assume that the Arthur parameter of $\pi$ factors through $\hat{\iota}':\hat{H}'(\BC)\times \SL_2(\BC)\rightarrow \hat{G}(\BC)$ (i.e. $\pi$ is a lifting of a global tempered Arthur packet $\Pi$ of $H'(\BA)$ where $H'$ is the Langlands dual group of $\hat{H}'$). Then we can choose the embedding $\nu$ so that 
$$\frac{\CP_{H,\iota,\rho_H}(\phi,\varphi)\cdot \CP_{\iota'}(\overline{\phi})}{\pair{\phi,\phi}}``=" \frac{1}{L(1,\Pi,Ad)},\;\phi\in Im(\nu),\;\varphi\in \CS(Y(\BA)).$$
\end{conj}

For the rest of this subsection, motivated by the above conjecture, we will propose a relative trace formula comparison to study the period integrals $\CP_{H,\iota,\rho_H}$ and $\CP_{\iota'}$. One side of the relative trace formula is just the Kuznetsov  trace formula for $H'(\BA)$. To be specific, let $N'$ be a maximal unipotent subgroup of $H'$, $\xi'$ be a generic character of $N'(k)\back N'(\BA)$, $f'$ be a Schwartz function on $H'(\BA)$ and $K_{f'}(x,y)$ be the usual kernel function. Then the Kuznetsov relative trace formula is given by
$$J(f')=\int_{N'(k)\back N'(\BA)}\int_{N'(k)\back N'(\BA)} K_{f'}(n_1,n_2)\xi'(n_1^{-1}n_2)\ dn_1\ dn_2.$$
Spectrally, the discrete part of $J(f')$ is of the form
$$J_{disc}(f')=\sum_{\Pi} J_{\Pi}(f')$$
where $\Pi$ runs over all the global discrete generic Arthur packet of $H'(\BA)$ and
$$J_{\Pi}(f')=\sum_{\phi} \int_{N'(k)\back N'(\BA)} \Pi(f)\phi(n)\xi'(n)^{-1}\ dn\ \overline{\int_{N'(k)\back N'(\BA)} \phi(n)\xi'(n)^{-1}\ dn}.$$
Here $\phi$ runs over an orthonormal basis of $\Pi$. According to Lapid-Mao's conjecture for the Whittaker period \cite{LM}, 
$$J_{\Pi}(f')``=" \frac{1}{L(1,\Pi,Ad)}.$$
Hence the discrete part of $J(f')$ is of the form
\begin{equation}\label{spectral expansion 1}
J_{disc}(f')``="\sum_{\Pi} \frac{1}{L(1,\Pi,Ad)}.
\end{equation}

For the other side of the relative trace formula, let $f$ be a Schwartz function of $G(\BA)$ and $K_f(x,y)$ be the usual kernel function. Then the other side of the relative trace formula is given by taking the period integral $\CP_{H,\iota,\rho_H}$ on the first variable of $K_f(x,y)$ and taking the period integral $\CP_{\iota'}$ on the second variable, i.e. we define
$$I(f)=\CP_{\iota'}(\CP_{H,\iota,\rho_H,1}(K_f)),\;\text{where  }\CP_{H,\iota,\rho_H,1}(K_f)(y):=\CP_{H,\iota,\rho_H,1}(K_f(\cdot, y)).$$

Note that we made an abuse of notation here: the distribution denoted $I(f)$ above may also depend on the Schwartz functions used in the definition of the periods $\CP_{H,\iota,\rho_H,1}$ and $\CP_{\iota'}$.

By Conjecture \ref{conj rtf period 1}, the discrete part of this relative trace formula should also have the spectral expansion \eqref{spectral expansion 1}. Thus we expect the following conjectural comparison between these two relative trace formulas.

\begin{conj}\label{conj RTF}
There should be a comparison between the above two relative trace formulas $I(f)$ and $J(f')$.
\end{conj}

The notion of a comparison between two relative trace formulas was introduced by Jacquet; we refer to \cite{J, JY} for some examples of comparisons of relative trace formulas.

In this paper, we provide some evidence for the above conjecture. More specifically, the comparisons of relative trace formulas in Section 3-5 of this paper are special cases of Conjecture \ref{conj RTF} when $(G,H,\rho_H,\iota)$ is the dual of the models in Table \ref{fig:1}. In all the cases, $H'$ is the Langlands dual of $H_0=\PGL_2$, therefore $H'=\SL_2$.

When $(G,H,\rho_H,\iota)$ is the Shalika model of $G=\GL_{2n}$ ($H=\{diag(h,h)|\;h\in \GL_n\}$, $\iota(\begin{pmatrix}1&1\\0&1\end{pmatrix})=\begin{pmatrix}I_n&I_n\\0&I_n\end{pmatrix}$ and $\rho_H=0$), the relative trace formula comparison in Conjecture \ref{conj RTF} recovers the one in \cite{FJ}, where $H'=\SO_{2n+1}$. In case $n=2$, the fundamental lemma for this relative trace formula is proved in \cite{FJ, MR5}.

When $(G,H,\rho_H,\iota)$ corresponding to the symmetric spaces of skew-symmetric matrices (with $G=GL_{2n}$, $H=\Sp_{2n}$, $\rho_H=0$ and $\iota=1$), we have $H'=\GL_n$ and the relative trace formula comparison is established in \cite{JR}.

When $(G,H,\rho_H,\iota)=(\GL_n,\GL_{n-1},0,1)$ with $n>2$, we have $H'=\GL_2$ and the relative trace formula in Conjecture \ref{conj RTF} recovers the one in (2.34) of \cite{MR4}. One can prove the comparison following the same ideas in \cite{MR4}. We will consider this case and provide more examples of Conjecture \ref{conj RTF} in a future paper. We would like to emphasize that if one can prove the relative trace formula comparison in Conjecture \ref{conj RTF} for some specific quadruples $(G,H,\rho_H,\iota)$ and $(\hat{G},\hat{H}',\rho_{\hat{H}'},\hat{\iota}')$, it will be strong evidence for Conjecture \ref{conj Whittaker general} of the degenerate Whittaker period $\CP_{\iota'}$ (this also applies to the comparison in Conjecture \ref{conj RTF rank 1}).

\subsection{A relative trace formula for rank 1 spherical varieties}\label{sec RTF2}
In this subsection we propose a relative trace formula comparison for all the rank 1 spherical varieties. Let $(G,H,\rho_H,\iota)$ and $(\hat{G},\hat{H}',\rho_{\hat{H}'},\hat{\iota}')$ be two quadruples that are dual to each other under the BZSV duality. We make the following assumption.

\begin{assu}\label{assumption 2}
$\iota$ is trivial, $\rho_H$ is zero-dimensional, and $G/H$ is a rank 1 spherical variety (Table (3) of \cite{S}).
\end{assu}

We use the same notation as in the previous subsection. In particular, we still have the decomposition
$$\hat{\Fg}=\oplus_{k\in \hat{I}} \hat{\rho}_k\otimes Sym^k$$
and we still want to study the period integrals $\CP_{H,\iota,\rho_H}$ and $\CP_{\iota'}$. According to Table (3) of \cite{S}, under Assumption \ref{assumption 2}, we have $\hat{H}'=\PGL_2$ or $\SL_2$. Moreover, it also implies that under Assumption \ref{assumption 2}, the quadruple $(G,H,\rho_H,\iota)$ satisfies Assumption \ref{assumption 1} if and only if $\rho_{\hat{H}'}=0$. In particular, for the models in Table (3) of \cite{S}, only the following four models of $(G,H)$ do not satisfy Assumption \ref{assumption 1}:
\begin{equation}\label{rank 1 models}
(\PGL_2,\GL_1), (\SO_{2n+1},\SO_{2n}), (\Sp_4,\Sp_2\times \Sp_2), (G_2,\SL_3).
\end{equation}

If the quadruple also satisfies Assumption \ref{assumption 1}, we can just use the relative trace formula comparison in Conjecture \ref{conj RTF} to study the period integrals $\CP_{H,\iota,\rho_H}$ and $\CP_{\iota'}$, i.e., there should be a comparison between the relative trace formula of the period integrals $\CP_{H,\iota,\rho_H}$ and $\CP_{\iota'}$ with the Kuznetsov relative trace formula for $H'(\BA)$.

For the rest of this subsection, we consider the case when Assumption \ref{assumption 1} fails, i.e., the four models in \eqref{rank 1 models}. In this case, $\hat{H}'=\SL_2$ and $H'=\PGL_2$. As in the previous subsection, one side of our relative trace formula is still
$$I(f)=\CP_{\iota'}(\CP_{H,\iota,\rho_H,1}(K_f)),\;\CP_{H,\iota,\rho_H,1}(K_f)(y):=\CP_{H,\iota,\rho_H,1}(K_f(\cdot, y)).$$

In this case, by Table (3) of \cite{S}, Conjecture \ref{BSV conj} and \ref{conj Whittaker general}, we expect the following conjecture.

\begin{conj}
Let $\pi$ be an irreducible discrete automorphic representation of $G(\BA)$ and let $\nu:\pi\rightarrow L^2(G(k)\back G(\BA))_{\pi}$ be an embedding. Assume that the Arthur parameter of $\pi$ factors through $\hat{\iota}':\hat{H}'(\BC)\times \SL_2(\BC)\rightarrow \hat{G}(\BC)$ (i.e. $\pi$ is a lifting of a global tempered Arthur packet $\Pi$ of $H'(\BA)=\PGL_2(\BA)$). Then we can choose the embedding $\nu$ so that 
$$\frac{\CP_{H,\iota,\rho_H}(\phi,\varphi)\cdot \CP_{\iota'}(\overline{\phi})}{\pair{\phi,\phi}}``=" \frac{L(1/2,\Pi)}{L(1,\Pi,Ad)},\;\phi\in Im(\nu),\;\varphi\in \CS(Y(\BA)).$$
\end{conj}

Based on the above conjecture, as in the previous subsection, the discrete part of the spectral expansion of $I(f)$ should be of the form
\begin{equation}\label{spectral expansion 2}
I_{disc}(f)``="\sum_{\Pi} \frac{L(1/2,\Pi)}{L(1,\Pi,Ad)}.
\end{equation}
where $\Pi$ runs over all the global discrete generic Arthur packet of $H'(\BA)=\PGL_2(\BA)$.

For the other side of the relative trace formula, let $f'$ be a Schwartz function of $H'(\BA)=\PGL_2(\BA)$ and let $K_{f'}(x,y)$ be the kernel function. Let $T=\{diag(a,1)\}\subset H'$, $N'$ be a maximal unipotent subgroup of $H'$, and $\xi'$ be a generic character of $N'(k)\back N'(\BA)$. Then the other side of the relative trace formula is given by 
$$J(f')=\int_{N'(k)\back N(\BA)}\int_{T(k)\back T(\BA)}K_{f'}(x,y)\xi'(x)^{-1}dydx$$
i.e., we take the Whittaker period on the first variable and the Hecke period on the second variable. This relative trace formula was first considered by Jacquet in \cite{J}. It is well known that the discrete part of the spectral expansion of this relative trace formula is also equal to \eqref{spectral expansion 2}. This leads to the following conjectural comparison between these two relative trace formulas.

\begin{conj}\label{conj RTF rank 1}
There should be a comparison between the above two relative trace formulas $I(f)$ and $J(f')$.
\end{conj}

The above conjecture is an analog of comparing rank 1 spherical varieties in \cite{S}. Thanks to our conjecture of degenerate Whittaker model in Conjecture \ref{conj Whittaker general}, we expect the comparison in Conjecture \ref{conj RTF rank 1} to be a point-to-point comparison instead of the beyond endoscopic type comparison in \cite{S}, also the test functions in this comparison are the standard Schwartz functions instead of some generalized Schwartz functions as in \cite{S}. We will prove this comparison in a future paper. 

\begin{rmk}
Among the four models in \eqref{rank 1 models}, Conjecture \ref{conj RTF rank 1} is trivial when $(G,H)=(\PGL_2,\GL_1)$. When $(G,H)=(\Sp_4,\Sp_2\times \Sp_2)$, we will explain in Section 6 that Conjecture \ref{conj RTF rank 1} essentially follows from the result in \cite{MR3}. The remaining  models $(\SO_{2n+1},\SO_{2n})$ and $(G_2,\SL_3)$ will be considered in a future paper. We note that the fundamental lemma in the $(\SO_5,\SO_4)$ case is already established in \cite{Z}.
\end{rmk}

\subsection{Organization of the paper}
In Section 2, we will explicitly state the BZSV conjecture for models in Table \ref{fig:1}. In Sections 3-5, we will consider the relative trace formula for the dual side of these models. More specifically, in Section 3, we will recall the Kuznetsov relative trace formula for $\SL_2$; in Section 4 (resp. Section 5), we will consider the relative trace formula for Model 1-4 (resp. Model 5-6), and we will prove the fundamental lemma and smooth transfer. In Section 6, we will provide more examples of Conjecture \ref{conj Whittaker general}.

\subsection{Acknowledgement} We thank Yiannis Sakellaridis and Akshay Venkatesh for explaining the BZSV duality to us and for the helpful comments on earlier drafts of this paper. We thank Wee Teck Gan for helpful discussions. The work of the first author is partially supported by the Simons Collaboration Grant. The second author's work is partially supported by the NSF grant DMS-2000192 and DMS-2103720. 
The work of the third author is partially supported by AcRF Tier 1 grants 	A-0004274-00-00 and A-0004279-00-00 of the National University of Singapore.

\section{BSV conjecture for the models in Table \ref{fig:1}} \label{sec DMVexplicit}
In this section, we will explicitly write down the BZSV conjectures for the models in Table \ref{fig:1}. We first recall the table.

\begin{figure}[h!]
\begin{tabular}{| c | c | c |c|c| c|}
\hline
\textnumero & $G$ & $H=H_0\ltimes U$ & $\hat{G}$ &  $\hat{\rho}$\\
\hline
1 &  $\PGL_6$ & $\PGL_2\ltimes U$ & $\SL_6$ & $\wedge^3$   \\
\hline
2 & $\GSO_8\times \GL_2/\GL_1$ & $\PGL_2\ltimes U$ & $S(\GSpin_8\times \GL_2)$ & $\HSpin_8\otimes {\rm std}_{2}$    \\
\hline
3 & $\PGSO_{12}$ & $\PGL_2\ltimes U$  & $\Spin_{12}$ & $\HSpin_{12}$  \\
\hline
4 & $E_{7,ad}$ & $\PGL_2\ltimes U$ & $E_{7,sc}$  &$\omega_7$ \\
\hline
5 & $\PGSp_{10}$ & $\PGL_2\ltimes U$ & $\Spin_{11}$ & $\Spin_{11}$   \\
\hline
6 & $\GSp_{6}\times \GL_2/\GL_1$ & $\PGL_2\ltimes U$ & $S(\GSpin_7\times \GL_2)$ & $\Spin_{7}\otimes {\rm std}_2$  \\
\hline
\end{tabular}
\captionof{table}{}
\label{fig:2}
\end{figure}

All the models in the table are the Whittaker induction of the trilinear $\GL_2$-model. In each case, $G$ has a parabolic subgroup $P=MU$ whose Levi factor $M$ is of Type $A_1\times A_1\times A_1$. We can define a generic character $\xi$ of $U(F)$ (or $U(k)\back U(\BA)$ globally) such that the stabilizer of $\xi$ in $M$ (under the adjoint action) is just $H_0$. If we fix an additive character $\psi$ of $F$ (or $\BA/k$ in the global case), then there exists an element $\Xi\in \bar{\Fu}(F)$ (or $\bar{\Fu}(k)$; here $\bar{P}=M\bar{U}$ is the opposite parabolic subgroup and $\bar{\Fu}$ is the Lie algebra of $\bar{U}$) such that
$$\xi(\exp(X))=\psi(\pair{\Xi,X})),\;X\in \Fu(F)\; (\text{or } \Fu(\BA))$$
where $\pair{,}$ is the Killing form on the Lie algebra.

In the first case, if we use $\alpha_i=e_i-e_{i+1}$ ($1\leq i\leq n$) to denote the simple roots of Type $A_n$, then $P$ is the standard parabolic subgroup associated to the simple roots $\alpha_1,\alpha_3,\alpha_5$ and $\Xi=X_{-\alpha_1-\alpha_2}+X_{-\alpha_2-\alpha_3}+X_{-\alpha_3-\alpha_4}+X_{-\alpha_4-\alpha_5}$ where $X_\beta$ is a non-zero element in the root space of $\beta$.

In the second case, if we use $\alpha_i=e_i-e_{i+1}$ ($1\leq i\leq n-1$) and $\alpha_n=e_{n-1}+e_n$ to denote the simple roots of Type $B_n$, then $P$ is the product of $\GL_2$ with the standard parabolic subgroup of $\GSO_8$ associated to the simple roots $\alpha_1,\alpha_3$ and $\Xi=X_{-\alpha_1-\alpha_2}+X_{-\alpha_2-\alpha_3}+X_{-\alpha_4}$.

In the third case, $P$ is the standard parabolic subgroup of $\GSO_{12}$ associated to the simple roots $\alpha_1,\alpha_3,\alpha_5$ and $\Xi=X_{-\alpha_1-\alpha_2}+X_{-\alpha_2-\alpha_3}+X_{-\alpha_3-\alpha_4}+X_{-\alpha_4-\alpha_5}+X_{-\alpha_6}$.

In the fourth case, consider the following Dynkin diagram for $E_7$:
\begin{figure}[h!]
\begin{tikzpicture}[inner sep=0.5mm,scale=0.75]
\node [circle,draw,label=below:${\alpha_1}$] (1)  at ( 0,0)  {}; 
\node [circle,draw,label=below:${\alpha_3}$] (3) at ( 1,0) {};
\node [circle,draw,label=below:${\alpha_4}$] (4) at ( 2,0) {};
\node [circle,draw,label=below:${\alpha_5}$] (5) at ( 3,0) {}; 
\node [circle,draw,label=below:${\alpha_6}$] (6)  at (4,0) {};
\node [circle,draw,label=below:${\alpha_7}$] (7) at ( 5,0) {};
\node [circle,draw,label=right:${\alpha_2}$] (2)  at ( 2,1) {}; 
\draw  (1) -- (3);
\draw  (3) -- (4);
\draw  (2) -- (4);
\draw  (4) -- (5);
\draw  (5) -- (6);
\draw  (6) -- (7);
\end{tikzpicture}
\end{figure}
$P$ is the standard parabolic subgroup of $E_7$ associated to the simple roots $\alpha_2,\alpha_5,\alpha_7$ and $\Xi=X_{-\alpha_2-\alpha_4}+X_{-\alpha_4-\alpha_5}+X_{-\alpha_5-\alpha_6}+X_{-\alpha_6-\alpha_7}+X_{-\alpha_1}+X_{-\alpha_3}$.

In the fifth case, if we use $\alpha_i=e_i-e_{i+1}$ ($1\leq i\leq n-1$) and $\alpha_n=2e_n$ to denote the simple roots of Type $C_n$, then $P$ is the standard parabolic subgroup of $\GSp_{10}$ associated to the simple roots $\alpha_1,\alpha_3,\alpha_5$ and $\Xi=X_{-\alpha_1-\alpha_2}+X_{-\alpha_2-\alpha_3}+X_{-\alpha_3-\alpha_4}+X_{-\alpha_4-\alpha_5}$.

In the last case, $P$ is the product of $\GL_2$ with the standard parabolic subgroup of $\GSp_6$ associated to the simple roots $\alpha_1,\alpha_3$ and $\Xi=X_{-\alpha_1-\alpha_2}+X_{-\alpha_2-\alpha_3}$.

In each case, $\Xi$ induces an embedding (denoted by $\iota$) from $\SL_2$ into $G$ that maps $\begin{pmatrix}1&0\\1&1\end{pmatrix}$ to $\exp(\Xi)$. Let $T_{\SL_2}$ be the diagonal torus of $\SL_2$. Then it is easy to see that the centralizer of $Im(\iota(T))$ in $G$ is just $M$ and $H_0$ commutes with $Im(\iota)$. This gives us an embedding $\iota:H_0\times \SL_2\rightarrow G$.

For each of the models in Table \ref{fig:2}, in terms of the language of BZSV duality, one of the quadruple is just 
$$(G,H_0,0,\iota).$$
The other quadruple is 
$$(\hat{G},\hat{G},\hat{\rho},1)$$
where $\hat{\rho}$ is given in Table \ref{fig:2} and $1$ is the map from $\SL_2$ to $\hat{G}$ which sends all the elements to identity. 

Next we discuss the sets $I,\hat{I}$ and the representations $\rho_k,\hat{\rho}_k$. Since the $\SL_2$-embedding in the dual side is trivial, we have $\hat{I}=\{0\}$ and $\hat{\rho}_0$ is the adjoint representation of $\hat{G}$. The set $I$ and the representations $\rho_k$ are given by the following table (here we use $std_2$ to denote the standard representation of $H_0=\PGL_2$).

\begin{figure}[h!]
\begin{tabular}{| c | c | c |c|c| c|}
\hline
\textnumero &  $I$ & $\rho_k$ \\
\hline
1 &   $\{0,2,4\}$ & $\rho_0=\rho_2=\rho_4=std_2$  \\
\hline
2 &  $\{0,4\}$ & $\rho_0=std_2\oplus std_2,\;\rho_4=std_2$  \\
\hline
3   & $\{0,4,8\}$ & $\rho_0=\rho_4=\rho_8=std_2$ \\
\hline
4 &    $\{0,8,16\}$ & $\rho_0=\rho_8=\rho_{16}=std_2$ \\
\hline
5 &   $\{0,4,8\}$ & $\rho_0=\rho_4=\rho_8=std_2$ \\
\hline
6 & $\{0,4\}$ & $\rho_0=std_2\oplus std_2,\;\rho_4=std_2$ \\
\hline
\end{tabular}
\captionof{table}{}
\label{fig:2}
\end{figure}

From the table above, we can see that the quadruple $(\hat{G},\hat{G},\hat{\rho},1)$ satisfies Assumption \ref{assumption 1} in Section \ref{sec RTF}. In Sections 3-5, we will prove the fundamental lemma and smooth transfer of the relative trace formula in Conjecture \ref{conj RTF} for these 6 cases. Now, we can state the BZSV duality for these 6 models explicitly. We start from the $(G,H)$-side.

\begin{conj}\label{main conj}
Let $\pi$ be a generic cuspidal automorphic representation of $G(\BA)$. We can find an embedding $\nu:\pi\rightarrow \CA_{cusp}(G(\BA))_\pi$ such that
$$\frac{|\CP_{H_0,\iota,0}(\phi)|^2}{\pair{\phi,\phi}}``=" \frac{L(\frac{1}{2},\pi,\hat{\rho})}{L(1,\pi,Ad)},\;\phi\in \nu(\pi).$$
\end{conj}

\begin{rmk}
In this case, the period integral $\CP_{H_0,\iota,0}(\phi)$ is given by 
$$\CP_{H_0,\iota,0}(\phi)=\int_{H_0(k)\back H_0(\BA)}\int_{U(k)\back U(\BA)}\phi(uh_0)\xi(u)dudh_0.$$
\end{rmk}

We refer the reader to Conjecture 1.7 of \cite{WZ} for a more explicit version of Conjecture \ref{main conj} (i.e., the one that includes the local relative character, all the Dedekind zeta functions, and all the constants). Next we state the conjecture on the dual side. Recall that we have a symplectic representation $\hat{\rho}:\hat{G}\rightarrow \Sp(V)$. Let $Y$ be a maximal isotropic subspace of $V$ and $\Omega_\psi$ be the Weil representation of $\widetilde{\Sp}(V)$ on the Schwartz space $\CS(Y(\BA))$. It is easy to see that $\widetilde{\Sp}(V)$ splits over $Im(\hat{\rho})$ and hence $\Omega_\psi$ restricts to a representation of $\hat{G}(\BA)$ on $\CS(Y(\BA))$.
% which will be denoted by $\rho_{\hat{G},\psi}$ (or sometimes just  $\Omega_\psi$ if the choice of $\hat{G}$ is clear from the context). 
We define the theta series
$$\Theta_{\psi}^{\varphi}(g)=\sum_{X\in Y(k)} \Omega_{\psi}(g)\varphi(X),\;g\in G(\BA),\varphi\in \CS(Y(\BA)).$$

\begin{conj}\label{main conj dual}
Let $\pi$ be an irreducible discrete automorphic representation of $\hat{G}(\BA)$ with trivial central character and let $\nu:\pi\rightarrow L^2(\hat{G}(k)\back \hat{G}(\BA)/Z_{\hat{G}}(\BA))_{\pi}$ be an embedding. Then the period integral
$$\CP_{\hat{G},1,\hat{\rho}}(\phi,\varphi)=\int_{\hat{G}(k)Z_{\hat{G}}(\BA)\back \hat{G}(\BA)} \phi(g)\Theta_{\psi}^{\varphi}(g)\  dg,\;\phi\in \nu(\pi),\;\varphi\in \CS(Y(\BA))$$
is nonzero only if the Arthur parameter of $\pi$ factors through $\iota:H_0(\BC)\times \SL_2(\BC)\rightarrow G(\BC)$. If this is the case, then $\pi$ is a lifting of a global tempered Arthur packet $\Pi$ of $\SL_2(\BA)$ (recall that $H_0=\PGL_2$) and we can choose the embedding $\nu$ so that 
$$\frac{|\CP_{\hat{G},1,\hat{\rho}}(\phi,\varphi)|^2}{\pair{\phi,\phi}}``=" \frac{\Pi_{k\in I}L(k/2+1,\Pi,\rho_k)}{L(1,\Pi,Ad)^2}.$$
\end{conj}

In the following three sections, we will study the relative trace formula comparison in Conjecture \ref{conj RTF} for the quadruple $(\hat{G},\hat{G},\rho,1)$ where $(\hat{G},\hat{\rho})$ runs over models in Talbe \ref{fig:2}, and we will prove the fundamental lemma and smooth transfer in the p-adic case. This serves as strong evidence for Conjecture \ref{main conj dual}.

\begin{rmk}
The functorial lifting from $\SL_2$ to $\hat{G}$ where $G$ runs over the groups in Table \ref{fig:1} can also be proved using theta correspondence method based on the minimal representations constructed by Kazhdan-Savin (\cite{KS}, \cite{GRS1}, \cite{MS}).

\end{rmk}

\section{The relative trace formula on $\SL_2$}
As we explained in Section \ref{sec RTF} when we consider the relative trace formula for quadruples $(\hat{G},\hat{G},\hat\rho,1)$ defined in the previous section, one side of the relative trace formula is just the Kuznetsov trace formula for $H'=\SL_2$. To be specific, let $G'=H'=\SL_2$ and $N'$ be the upper triangular unipotent subgroup of $G'$ and we define the character $\xi'$ on $N'(k)\back N'(\BA)$ to be $\xi'(\begin{pmatrix}1&x\\0&1\end{pmatrix})=\psi(x)$. For $f'\in \CS(G'(\BA))$, let 
$$K_{f'}(x,y)=\sum_{\gamma\in G'(k)}f(x^{-1}\gamma y)$$ 
be the kernel function and define
$$J(f')=\int_{N'(k)\back N'(\BA)}\int_{N'(k)\back N'(\BA)}K_{f'}(n_1,n_2)\psi(n_1)^{-1}\psi(n_2)\ dn_1\ dn_2.$$
Here we can identify $N'(\BA)$ with $\BA$, and the measure on it comes from the $\psi$-self dual Haar measure on $\BA$ (thus the total volume of $\BA/k$ equal to 1).

If $f'=\otimes_{v\in |k|} f_v'$, for $a\in k_{v}^{\times}$, define
$$J_v(a,f_v')=\int_{k_v}\int_{k_v} f_v'(\begin{pmatrix}1&x\\0&1\end{pmatrix} \begin{pmatrix}0&-a^{-1}\\a&0  \end{pmatrix} \begin{pmatrix}1&y\\0&1\end{pmatrix})\psi(x+y)\ dx\ dy,$$
$$J_{v}^{+}(f_v')=\int_{k_v}f_v'(\begin{pmatrix}1&x\\0&1\end{pmatrix})\psi(x)\ dx,\;J_{v}^{-}(f_v')=\int_{k_v}f_v'(\begin{pmatrix}-1&-x\\0&-1\end{pmatrix})\psi(x)\ dx.$$
A standard unfolding process (see \cite{JY}) gives the identity
\begin{equation}\label{unfolding Whittaker}
J(f')=\Pi_{v\in k^\times} J_{v}^{+}(f_v')+\Pi_{v\in k^\times} J_{v}^{-}(f_v')+\sum_{a\in k^\times}\Pi_{v\in k^\times} J_v(a_v,f_v').
\end{equation}

For the rest of this section, we recall the properties of the orbital integral $J_v(a_v,f_v')$ in the p-adic case. To simplify the notation, we will use $F$ to replace the local field $k_v$, and we will get rid of all the subscript $v$. We also assume that $F$ is non-archimedean. Let $\CO_F$ be the ring of integers, $\varpi$ be a uniformizer, $p$ be the residue characteristic and $q=|\varpi|^{-1}$. We also fix an additive character $\psi$ of $F$. We say $\psi$ is unramified if it is trivial on $\CO_F$ but not trivial on $\varpi^{-1}\CO_F$. We choose the self-dual Haar measure $da$ on $F$ (with respect to $\psi$). If $\psi$ is unramified, then the volume of $\CO_F$ under $da$ equals 1.

\begin{defn}\label{space of orbital integral}
For $c\in \BC$, define $C_{OI}^{\infty}(F^\times,c_+,c_-)$ to be the space of locally constant complex valued functions $f$ on $F^\times$ such that
\begin{enumerate}
\item The support of $f$ is bounded on $F$.
\item There exists a neighborhood $\omega$ of $0$ in $F$ such that
$$f(a)=c_+\cdot \int_{1+\varpi^m\CO_{F}}\psi(\frac{x+x^{-1}}{a})\ dx+c_-\cdot \int_{-1+\varpi^m\CO_{F}}\psi(\frac{x+x^{-1}}{a})\ dx,\;\forall a\in \omega-\{0\}.$$
\end{enumerate}
Here we choose $m>0$ large enough so that $|2|>|\varpi^m|$ (note that the two integrals are independent of the choice of $m$ when $a$ is close to 0). We let $C_{OI}^{\infty}(F^{\times})=\cup_{c_i\in \BC} C_{OI}^{\infty}(F^\times,c_1,c_2)$.
\end{defn}

\begin{rmk}
Here $OI$ stands for orbital integral. This space is the same as the one in \cite{S} for the orbital integrals of rank one spherical varieties. Also when $c_1=c_2=0$, the space $C_{OI}^{\infty}(F^\times,c_1,c_2)$ is just $C_{c}^{\infty}(F^\times)$.
\end{rmk}

The following proposition about the orbital integrals of the Kuznetsov relative trace formula for $\SL_2$ is well known.

\begin{prop}\label{orbital integral SL(2)}
\begin{enumerate}
\item For $F\in C_{c}^{\infty}(F^\times)$, there exists $f'\in C_{c}^{\infty}(G'(F))$ such that
$$J(a,f')=F(a),\;\forall a\in F^\times.$$
\item If $f'$ is the characteristic function of $G'(\CO_F)$ and $\psi$ is unramified, then $J^+(f')=J^-(f')=1$ and
$$J(a,f')=0\;\text{if}\;|a|>1;\;J(a,f')=1\;\text{if}\;|a|=1;\;J(a,f')=|a|^{-1}\cdot \int_{\CO_{F}^{\times}}\psi(\frac{x+x^{-1}}{a})\ dx\;\text{if}\;|a|<1.$$
\item For $f'\in C_{c}^{\infty}(G'(F))$, the function
$$a\in F^\times\mapsto |a|\cdot J(a,f')$$
belongs to the space $C_{OI}^{\infty}(F^\times,c_+,c_-)$ for $c_+=J^+(f')$ and $c_-=J^-(f')$.
\end{enumerate}
\end{prop}

\section{The relative trace formula for Model 1-4}
In this section, we will write down the other side of the relative trace formula for Models 1-4. We will also unfold the relative trace formula and formulate the conjecture of fundamental lemma and smooth transfer. Then we will prove the fundamental lemma as well as the smooth transfer in the p-adic case.

\subsection{The relative trace formula}

In this case, the relative trace formula in Section \ref{sec RTF} has already been formulated in a previous paper of the first author and S. Rallis \cite{MR} for Models 1, 3, and 4. The trace formula for Model 2 is similar. Let's recall the relative trace formula from \cite{MR}. We first need to specify a parabolic subgroup $Q=LN$ of $\hat{G}$. For Model 1 (resp. Model 3, Model 4), $Q$ is the standard maximal parabolic subgroup whose Levi factor does not contain the simple root $\alpha_3$ (resp. $\alpha_6$,  $\alpha_7$). For Model 2, $Q$ is the product of the standard Borel subgroup of $\GL_2$ with the standard maximal parabolic subgroup of $\GSO_8$ whose Levi factor does not contain the simple root $\alpha_4$.

\begin{rmk}
The Richardson nilpotent orbit associated to $Q$ is just the dual of the nilpotent orbit $\CO_\iota$ of $\Fg$.
\end{rmk}

In all four cases, $N$ can be identified with a Jordan algebra $\CJ$ of degree 3. More specifically, for Model 1 (resp. Model 2, Model 3, Model 4), $\CJ$ is $Mat_{3\times 3}$ (resp. $Asym_{4\times 4}\oplus Mat_{1\times 1}$, $Asym_{6\times 6}$, $3\times 3$ self-adjoint octonionic matrices). Here we use $Mat_{n\times n}$ (resp. $Asym_{n\times n}$) to denote the $n\times n$ matrices (resp. $n\times n$ skew symmetric matrices). 

In all these cases, we can define the trace map and norm map on $\CJ$ (denoted by $\tr$ and $N_\CJ$). As in \cite{MR}, for $A\in \CJ$ we can also define $A\cdot A,\;A\times A\in \CJ$ such that $(A\times A)\cdot A=N_\CJ(A)$. We define the generic character on $N(k)\back N(\BA)$ to be
$$\xi_N(A)=\psi(\tr(A)),\;A\in \CJ.$$

For $f\in \CS(\hat{G}(\BA)/Z_{\hat{G}}(\BA))$, let $K_f(x,y)$ be the kernel function as in the previous section. For $\varphi\in \CS(Y(\BA))$, we define
$$I(f,\varphi)=\int_{N(k)\back N(\BA)}\int_{\hat{G}(k)Z_{\hat{G}}(\BA)\back \hat{G}(\BA)} K_f(g,n) \Theta_{\psi}^{\varphi}(g)\xi_N(n)\ dg\ dn.$$
Here we can identify $N(\BA)$ with $\BA^n$ with $n=\dim(N)$ and the measure of $\BA$ induces a Haar measure of $N(\BA)$ with $\vol(N(k)\back N(\BA))=1$. This is the other side of the relative trace formula in Conjecture \ref{conj rtf period 1} for the current case.

As in \cite{MR}, we can identify the maximal isotropic space $Y$ with $Y=Mat_{1\times 1}\oplus \CJ$. In \cite{MR}, they proved the unfolding process of $I(f,\varphi)$ for Models 1, 3, and 4. A similar argument also works for the case for Model 2, and we will skip the details here. We have
$$I(f,\varphi)=\Pi_{v\in |k|}I_{v}^{+}(f_v\ast \varphi_v)+\Pi_{v\in |k|}I_{v}^{+}(f_v\ast \varphi_v)+\sum_{a\in k^\times} I_v(f_v\ast\varphi_v),\;f=\otimes f_v,\varphi=\otimes \varphi_v$$
where ($I$ is the identity element in the Jordan algebra)
$$f_v\ast \varphi_v=\int_{\hat{G}(k_v)/Z_{\hat{G}}(k_v)} f_v(g^{-1}) \Omega_{\psi}(g)\varphi_v\ dg,\;I_{v}^{\pm}(f_v\ast \varphi_v)=f_v\ast\varphi_v(0,\pm I),$$
$$I_v(a,f_v\ast\varphi_v)=\int_{\CJ(k_v)} f_v\ast \varphi_v(a,A) \psi(\frac{\tr(A)-N_\CJ(A)}{a})\ dA.$$

Now, we can state the fundamental lemma and smooth transfer in this case. 

\begin{conj}\label{conj case 1}
With the notation above, let $n=\dim(\CJ)$. The following holds.
\begin{enumerate}
\item (Fundamental lemma) Assume that $v$ is non-archimedean and $\psi_v$ is unramified over $v$. Let $\CO_v$ be the ring of integers of $k_v$ and let $f_{0,v}'$ (resp. $f_{0,v},\;\varphi_{0,v}$) be the characteristic function of $G'(\CO_{v})$ (resp. $\hat{G}(\CO_v)Z_{\hat{G}}(F)$, $Y(\CO_v)$). Then (note that $f_{0,v}\ast \varphi_{0,v}=\varphi_{0,v}$)
$$I_{v}^{+}(\varphi_{0,v})=J_{v}^{+}(f_{0,v}'),\;I_{v}^{-}(\varphi_{0,v})=J_{v}^{-}(f_{0,v}'),\;I_v(a,\varphi_{0,v})=|a|^{(n+1)/2}J_v(a,f_{0,v}'),\;\forall a\in k_{v}^{\times}.$$
\item (Smooth transfer) For any $f_v'\in \CS(G'(k_v))$ (resp. $f_v\in \CS(\hat{G}(k_v)/Z_{\hat{G}}(k_v))$ and $\varphi_v\in \CS(Y(k_v))$), there exists $f_v\in \CS(\hat{G}(k_v)/Z_{\hat{G}}(k_v))$ and $\varphi_v\in \CS(Y(k_v))$ (resp. $f_v'\in \CS(G'(k_v))$) such that
$$I_{v}^{+}(f_v\ast \varphi_v)=J_{v}^{+}(f_v'),\;I_{v}^{-}(f_v\ast \varphi_v)=J_{v}^{-}(f_v'),\;I_v(a,f_v\ast\varphi_v)=|a|^{(n+1)/2}J_v(a,f_v'),\;\forall a\in k_{v}^{\times}.$$
\end{enumerate}
\end{conj}

\subsection{The proof of the fundamental lemma and smooth transfer}
In this subsection, we will prove the fundamental lemma and smooth transfer in the p-adic case. To simplify the notation, we will use $F$ to replace the local field $k_v$ and eliminate all the subscript $v$. We also assume that $F$ is non-archimedean. By Proposition \ref{orbital integral SL(2)}, we only need to prove the following theorem.

\begin{thm}\label{orbital integral Model 1-4}
\begin{enumerate}
\item For $F\in C_{c}^{\infty}(F^\times)$, there exists $\varphi\in C_{c}^{\infty}(Y(F))$ such that
$$I(a,\varphi)=F(a),\;\forall a\in F^\times.$$
\item If $\varphi$ is the characteristic function of $Y(\CO_F)$ and $\psi$ is unramified, then $I^+(\varphi)=I^-(\varphi)=1$ and
$$I(a,\varphi)=0\;\text{if}\;|a|>1;\;I(a,\varphi)=1\;\text{if}\;|a|=1;\;I(a,\varphi)=|a|^{(n-1)/2}\cdot \int_{\CO_{F}^{\times}}\psi(\frac{x+x^{-1}}{a})\ dx\;\text{if}\;|a|<1.$$
\item For $\varphi\in C_{c}^{\infty}(Y(F))$, the function
$$a\in F^\times\mapsto |a|^{(1-n)/2}I(a,\varphi)$$
belongs to the space $C_{OI}^{\infty}(F^\times,c_+,c_-)$ for $c_+=I^+(\varphi)$ and $c_-=I^-(\varphi)$.
\end{enumerate}
\end{thm}

The first part of the theorem is trivial. For the second part, the cases for Models 1, 3, and 4 have already been proved in \cite{MR}. The exact same argument can prove the case for Model 2. It remains to verify the last part of the theorem. Without loss of generality, we may assume that $\psi$ is unramified. If $\{0\}\times \CJ$ does not intersect with the support of $\varphi$, it is clear that $I(a,\varphi)\in C_{c}^{\infty}(F^\times)$. Hence it is enough to consider the case when $\varphi$ is the product of the characteristic function of $\varpi^m\CO_F$ and the characteristic function of $A_0+\varpi^m\CJ(\CO_F)$ for some $A_0\in \CJ$ and $m>0$ large. 

We first show that if $\pm I$ does not belong to the support $A_0+\varpi^m\CJ(\CO_F)$, then the function $I(a,\varphi)$ belongs to $C_{c}^{\infty}(F^\times)$. In fact, in this case, by choosing $m$ large enough we may assume that $A\times A\neq I$ for all $A\in A_0+\varpi^m\CJ(\CO_F)$. With the same notation as in Section 2 of \cite{MR}, we can write $A$ as $\Pi_{\gamma\in \Pi_0'} x_\gamma$ where $\Pi'_0$ is a set of roots defined in \cite{MR}.

For a fixed $\gamma\in \Pi_0'$, we can write $\tr(A)-N_\CJ(A)$ as $x_\gamma P_\gamma(A)+Q_\gamma(A)$ where $P_\gamma(A)$ and $Q_\gamma(A)$ are polynomials in $\{x_{\gamma'}|\;\gamma'\neq \gamma\}$. It is easy to see that $P_\gamma(A)$ is just the $\gamma$-coordinate of $A\times A-I$. Since $A\times A\neq I$ for all $A\in A_0+\varpi^m\CJ(\CO_F)$, there exists $\gamma\in \Pi_0'$ and $C>0$ such that $|P_\gamma(A)|>C$ for all $A\in A_0+\varpi^m\CJ(\CO_F)$. This implies that for $a$ small enough, for any $x_{\gamma'}\in F$ with $\gamma'\not=\gamma$, the integration over the variable $x_\gamma$ will be zero; thus the orbital integral $I(a,\varphi)$ is equal to 0. This proves that the function $I(a,\varphi)$ belongs to $C_{c}^{\infty}(F^\times)$.

It remains to consider the case when $\varphi$ is the product of the characteristic function of $\varpi^m \CO_F$ and the characteristic function of $I+\varpi^m\CJ(\CO_F)$ or $-I+\varpi^m\CJ(\CO_F)$. The argument for both cases is the same, so we assume that $\varphi$ is the product of the characteristic function of $\varpi^m \CO_F$ and the characteristic function of $I+\varpi^m\CJ(\CO_F)$. We need an elementary identity whose proof is trivial.
\begin{equation}\label{basic equation 1}
\int_{\varpi^m\CO_F}\int_{\varpi^m\CO_F}\psi(bxy)\ dx\ dy=|b|^{-1},\;\forall |b|>q^{2m}.
\end{equation}
Then by the same argument as in the proof of \cite[Theorem 2]{MR} except that we replace the identity in Lemma 4 of loc. cit. by the identity above \footnote{Note that on Page 334 of loc. cit. after the change of variables $x_{\gamma_1}\rightarrow x_{\gamma_1}+(1-C_1)/x_{\alpha+\beta},\;x_{\delta_1}\rightarrow x_{\delta_1}+(1-D_1)/x_{\alpha+\beta}$ the integration domain of $x_{\gamma_1}$ (resp. $x_{\delta_1}$) changes from $1+\varpi^m\CO_F$ to $\varpi^m\CO_F$.},  we can show that when $a$ is close to 0, the orbital integral $I(a,\varphi)$ is equal to 
$$|a|^{(n-1)/2}\cdot \int_{1+\varpi^m\CO_{F}}\psi(\frac{x+x^{-1}}{a})\ dx=|a|^{(n-1)/2}\cdot \int_{1+\varpi\CO_{F}}\psi(\frac{x+x^{-1}}{a})\ dx.$$
This finishes the proof of the theorem.

\section{The relative trace formula for Model 5 and 6}
In this section, we will write down the other side of the relative trace formula for Models 5 and 6. We will also unfold the relative trace formula and formulate the conjecture of fundamental lemma and smooth transfer. Then we will prove the fundamental lemma as well as the smooth transfer in the p-adic case. These two cases are slightly more difficult than the previous four cases because we need to take the Fourier-Jacobi coefficient.

\subsection{The relative trace formula}

We first define a Fourier-Jacobi coefficient for $\Spin_{2n+1}$. Let $P_n=M_nU_n$ be the standard parabolic subgroup of $\Spin_{2n+1}$ whose Levi factor is associated to the positive roots $\Delta-\{e_1-e_2,e_{n-1}+e_n\}$. The unipotent group $U_n$ contains an abelian subgroup spanned by $U_\alpha$ for $\alpha\in \{e_1,e_i+e_j|\;1\leq i<j\leq n\}$ which will be denoted by $U_n'$. For $u\in U_n'$, we define $\lambda(u)$ to be the summation of the projection of $u$ into $U_{e_1},U_{e_2+e_5}$ and $U_{e_3+e_4}$ and we let $U_n''$ to be the kernel of $\lambda_n$. Then it is easy to see that $U_n/U_n''$ is a Heisenberg group of dimension $2n+1$, and we get a Weil representation of $U/U_{n}''$ on the space of Schwartz functions of a vector space $Y_n$ of dimension $n$, which will be denoted by $\Omega_{n,\psi}$. We define the theta series
$$\Theta_{\psi,n}^{\varphi}(g)=\sum_{X\in Y_n(k)} \Omega_{n,\psi}(g)\varphi(X),\;\varphi\in \CS(Y_n(\BA)),\;g\in U_n(\BA).$$
It is a function on $U_n(k)U_n''(\BA)\back U_n(\BA)$.

For Model 5, $\hat{G}=\Spin_{11}$ and we let $Q=LN$ be the parabolic subgroup $P_5$ of $\Spin_{11}$ and we let $\Theta_{\psi,N}^{\varphi}$ be the theta series $\Theta_{\psi,5}^{\varphi}$ defined above. For Model 6, $\hat{G}=S(\GSpin_7\times \GL_2)$ and we let $Q=LN$ be the parabolic subgroup whose projection to $\GSpin_7$ (resp. $\GL_2$) is the parabolic subgroup $P_3$ (resp. the Borel subgroup). We have defined the theta function $\Theta_{\psi,3}^{\varphi}$ on the unipotent radical of $P_3$, let $\Theta_{\psi,N}^{\varphi}$ be the product of this function with the additive character 
$$\xi(\begin{pmatrix}1&x\\0&1 \end{pmatrix})=\psi(x)$$
on the unipotent subgroup of $\GL_2$. We use $Y'$ to denote the space $Y_5$ (resp. $Y_3$) if we are in the case of Model 5 (resp. Model 6), and we use $\Omega_{\psi}'$ to denote the Weil representation. 

For $f\in \CS(\hat{G}(\BA)/Z_{\hat{G}}(\BA))$, let $K_f(x,y)$ be the kernel function as in the previous sections. For $\varphi\in \CS(Y(\BA))$ and $\varphi'\in \CS(Y'(\BA))$, we define (the measure on $N(\BA)$ is the Haar measure such that $\vol(N(k)\back N(\BA))=1$)
$$I(f,\varphi,\varphi')=\int_{N(k)\back N(\BA)}\int_{\hat{G}(k)Z_{\hat{G}}(\BA)\back \hat{G}(\BA)} K_f(g,u) \Theta_{\psi}^{\varphi}(g)\Theta_{\psi,N}^{\varphi'}(u)\ dg\ du.$$
This is the other side of the relative trace formula in Conjecture \ref{conj rtf period 1} for the current case.

Next, we will unfold the distribution $I(f,\varphi,\varphi')$. We first consider Model 5. As in the previous cases, we can identify $Y$ with $Y=Mat_{1\times 1}\oplus \CJ$ where $\CJ=Asym_{6\times 6}$ is a degree 3 Jordan algebra. Let $N_0$ be the unipotent radical of the Siegel parabolic subgroup (i.e., it contains the roots $e_i,e_i+e_j$). We can identify $\CJ$ with $\Fn_0=Lie(N_0)$ \footnote{In the previous case, as in \cite{MR}, one can identify $\CJ$ with the Lie algebra of the Siegel unipotent subgroup of $\Spin_{12}$ which contains the roots $e_i+e_j$ for $1\leq i,j\leq 6$, here we just need to replace the roots $e_i+e_6$ of $\Spin_{12}$ by the roots $e_i$ of $\Spin_{11}$.}. Then the conjugation action of $N$ on $\Fn_0$ induces an action of $N$ on $\CJ$ (denoted by $(u,A)\in N\times \CJ\mapsto u.A$). We fix a basis $X_{e_i},X_{e_i+e_j}$ of $\Fn_0=\CJ$. By a similar argument as in \cite{R}, we have the following description of the Weil representation $\Omega_\psi$ which is an analog of Lemma 1 of \cite{MR} \footnote{Here we only write down the action when $a\neq 0$, the case when $a=0$ will appear later when we discuss the singular terms.}
\begin{equation}\label{theta 1}
\Omega_{\psi}(u_{e_1-e_i}(t))\varphi(a,A)=\psi(\frac{N_\CJ(A)-N_\CJ(u_{e_1-e_i}(t).A)}{a})\varphi(a,u_{e_1-e_i}(t).A)=\varphi(a,u_{e_1-e_i}(t).A),
\end{equation}
\begin{equation}\label{theta 2}
\Omega_\psi(u_{e_i}(t))\varphi(a,A)=\psi(\frac{N_\CJ(A)-N_\CJ(atX_{e_i}+u_{e_i}(t).A)}{a})\varphi(a,atX_{e_i}+u_{e_i}(t).A),
\end{equation}
\begin{equation}\label{theta 3}
\Omega_\psi(u_{e_i+e_j}(t))\varphi(a,A)=\psi(\frac{N_\CJ(A)-N_\CJ(atX_{e_i+e_j}+u_{e_i+e_j}(t).A)}{a})\varphi(a,atX_{e_i+e_j}+u_{e_i+e_j}(t).A)
\end{equation}
for $\varphi\in \CS(Y(\BA))=\CS(\BA\oplus \CJ(\BA))$ and $(a,A)\in \BA\oplus \CJ(\BA)$ with $a\neq 0$. Here for a root $\alpha$ of $\hat{G}$, we let $\Fu_\alpha\subset \hat{\Fg}$ be the root space of $\alpha$, $U_\alpha=\exp(\Fu_\alpha)$ and we have a natural homomorphism $u_\alpha:\Fg\Fl_1\rightarrow U_\alpha$. This in particular gives an action of $N$ on $Y=F\oplus \CJ$ denoted by $(n,(a,A))\mapsto n.(a,A)$. Since the action fix the $a$-coordinate, we will denoted it by $(n,(a,A))\mapsto (a,(n,a). A)$ (i.e. it is a $\GL_1\times N$-action on $\CJ$).

On the other hand, we can identify $Y'$ with $\oplus_{2\leq i\leq 5}\Fu_{e_1-e_i}$ and we fix a basis $X_{e_1-e_i}$ of it. We can also set the isotropic space dual to $Y'$ to be $X_{e_i}+\frac{(-1)^i}{2}X_{e_1-e_{7-i}}$. Then it is easy to see that the Weil representation $\Omega_{\psi}'$ is given by
($\varphi'\in \CS(Y'(\BA))$, $Y=\sum Y_i X_{e_1-e_i}\in Y'(\BA)$)
\begin{equation}\label{theta 4}
\Omega_{\psi}'(u_{e_1-e_i}(t))\varphi'(Y)=\varphi'(Y+tX_{e_1-e_i}),\;\Omega_{\psi}'(u_{e_i}(t))\varphi'(Y)=\psi(tY_i)\varphi'(Y+\frac{(-1)^itX_{e_1-e_{7-i}}}{2}),\;2\leq i\leq 5;
\end{equation}
\begin{equation}\label{theta 5}
\Omega_{\psi}'(u_{e_1}(t))\varphi'(Y)=\Omega_{\psi}'(u_{e_2+e_5}(t))\varphi'(Y)=\Omega_{\psi}'(u_{e_3+e_4}(t))\varphi'(Y)=\psi(t)\varphi'(Y);
\end{equation}
\begin{equation}\label{theta 6}
\Omega_{\psi}'(u_{e_i+e_j}(t))\varphi'(Y)=\varphi'(Y),\;e_i+e_j\notin \{e_2+e_5,e_3+e_4\}.
\end{equation}
This, in particular, gives an action of $N$ on $Y'$ denoted by $(n,Y)\mapsto n\star Y$. From \eqref{theta 1}-\eqref{theta 6} we know that 
\begin{equation}\label{theta 7}
\sum_{A\in \CJ(k),Y\in Y'(k)} \varphi(a,A)\varphi'(Y)=\sum_{n\in N(k)}\Omega_\psi(n)\varphi(a,0)\Omega_{\psi}'(n)\varphi'(0),\;\forall a\in k^\times.
\end{equation}
Moreover, let $N=N_1N_2$ where $N_1$ is spanned by $U_{e_1-e_i}$ and $N_2$ is spanned by the rest root spaces. By \eqref{theta 1}-\eqref{theta 6} it is easy to see that 
\begin{equation}\label{theta 8}
\int_{N_1(\BA)}\Omega_\psi(u)\varphi(a,0)\Omega_{\psi}'( u)\varphi'(0) \ du=\int_{Y'(\BA)}\Omega_\psi(u)\varphi(a,0)\Omega_{\psi}'( u)\varphi'(Y)\ dY
\end{equation}
and 
\begin{eqnarray}\label{theta 9}
\Omega_\psi(n)\varphi(a,0)\Omega_{\psi}'(n)\varphi'(Y)&=&\psi(\frac{-N_\CJ((n,a).0)-\pair{n\star Y,(n,a).0}+\tr((n,a).0)}{a}-\frac{\pair{(n,a).0,(n,a).0}}{2a^2})\nonumber\\
&&\times \varphi(a,(n,a).0)\varphi'(n\star Y),\;n\in N_2(\BA), Y\in Y'(\BA),\;a\neq 0.
\end{eqnarray}
Here for $Y=\sum Y_iX_{e_1-e_i}\in Y'$ and $A=\sum A_i X_{e_i}+A_{ij}X_{e_i+e_j}\in \CJ$, we define
$$\pair{Y,A}=\sum_{2\leq i\leq 5} Y_i A_i,\;\pair{A,A}=A_2A_5+A_3A_4.$$

Let $f=\otimes_{v\in |k|}f_v,\;\varphi=\otimes_{v\in |k|}\varphi_v$ and $\varphi'=\otimes_{v\in |k|}\varphi_v'$. We have
\begin{eqnarray*}I(f,\varphi,\varphi')&=&\int_{N(k)\back N(\BA)}\int_{\hat{G}(k)Z_{\hat{G}}(\BA)\back \hat{G}(\BA)} K_f(g,u) \Theta_{\psi}^{\varphi}(g)\Theta_{\psi,N}^{\varphi'}(u)\ dg\ du\\
&=&\int_{N(k)\back N(\BA)}\int_{\hat{G}(k)Z_{\hat{G}}(\BA)\back \hat{G}(\BA)}\sum_{\gamma\in \hat{G}/Z_{\hat{G}}(k)}f(g^{-1}\gamma u)\cdot \sum_{X\in Y(k)}\Omega_\psi(g)\varphi(X)\cdot \sum_{Y\in Y'(k)}\Omega_{\psi}'(u)\varphi'(Y) \ dg \ du\\
&=&\int_{N(k)\back N(\BA)}\int_{\hat{G}(\BA)/Z_{\hat{G}}(\BA)}f(g^{-1} u)\cdot \sum_{X\in Y(k)}\Omega_\psi(g)\varphi(X)\cdot \sum_{Y\in Y'(k)}\Omega_{\psi}'(u)\varphi'(Y) \ dg\ du\\
&=&\int_{N(k)\back N(\BA)}\sum_{X\in Y(k)}\Omega_\psi(u)f\ast\varphi(X)\cdot \sum_{Y\in Y'(k)}\Omega_{\psi}'(u)\varphi'(Y)\ du\\
&=&\int_{N(k)\back N(\BA)}\sum_{a\in k^\times,A\in \CJ(k)}\Omega_\psi(u)f\ast\varphi(a,A)\cdot \sum_{Y\in Y'(k)}\Omega_{\psi}'(u)\varphi'(Y)\  du\\
&&+\int_{N(k)\back N(\BA)}\sum_{A\in \CJ(k)}\Omega_\psi(u)f\ast\varphi(0,A)\cdot \sum_{Y\in Y'(k)}\Omega_{\psi}'(u)\varphi'(Y)\ du\end{eqnarray*}
where
$$f\ast \varphi=\int_{\hat{G}(\BA)/Z_{\hat{G}}(\BA)}f(g^{-1})\Omega_\psi(g)\varphi\ dg.$$

We first deal with the main term. By \eqref{theta 7}-\eqref{theta 9}, we have
\begin{eqnarray*}
	&&
\int_{N(k)\back N(\BA)}\sum_{a\in k^\times,A\in \CJ(k)}\Omega_\psi(u)f\ast\varphi(a,A)\cdot \sum_{Y\in Y'(k)}\Omega_{\psi}'(u)\varphi'(Y)\ du\\
&=&\int_{N(k)\back N(\BA)}\sum_{a\in k^\times}\sum_{\gamma\in N(k)}\Omega_\psi(\gamma u)f\ast\varphi(a,0)\Omega_{\psi}'(\gamma u)\varphi'(0) \ du\\
&=&\sum_{a\in k^\times}\int_{N(\BA)}\Omega_\psi(u)f\ast\varphi(a,0)\Omega_{\psi}'( u)\varphi'(0)\ du\\
&=&\sum_{a\in k^\times}\int_{Y'(\BA)}\int_{N_2(\BA)}\Omega_\psi(u)f\ast\varphi(a,0)\Omega_{\psi}'( u)\varphi'(Y)\ du\ dY\\
&=&\sum_{a\in k^\times} \int_{Y'(\BA)}\int_{N_2(\BA)} \psi(\frac{\tr((a,u).0)-N_\CJ((a,u).0)-\pair{u\star Y,(a,u).0}}{a}-\frac{\pair{(n,a).0,(n,a).0}}{2a^2})\\
&& \times f\ast\varphi(a,(a,u).0)\varphi'(u\star Y)\ du\\
&=&\sum_{a\in k^\times} \int_{\CJ(\BA)}\int_{Y'(\BA)}\psi(\frac{\tr(A)-N_\CJ(A)-\pair{Y,A}}{a}-\frac{\pair{Y,A}}{2a^2}) f\ast\varphi(a,A)\varphi'(Y)\ dY\ dA\\
&=&\sum_{a\in k^\times} \Pi_{v\in |k|} I_v(a,f_v\ast \varphi_v,\varphi_v')
\end{eqnarray*}
where 
$$I_v(a,f_v\ast\varphi_v,\varphi_v')=\int_{\CJ(k_v)}\int_{Y'(k_v)}\psi_v(\frac{\tr(A)-N_\CJ(A)-\pair{Y,A}}{a}-\frac{\pair{A,A}}{2a^2}) f_v\ast\varphi_v(a,A)\varphi_v'(Y)\ dY\ dA.$$

Next, we study the singular term
$$\int_{N(k)\back N(\BA)}\sum_{A\in \CJ(k)}\Omega_\psi(u)f\ast\varphi(0,A)\cdot \sum_{Y\in Y'(k)}\Omega_{\psi}'(u)\varphi'(Y)\ du.$$
As in \eqref{theta 1}-\eqref{theta 3}, we have the following description of the Weil representation on singular terms, which is also an analog of Lemma 1 of \cite{MR}:
\begin{equation}\label{theta 10}
\Omega_\psi(u_{e_i+e_j}(t))\varphi(0,A)=\psi(-\tr((tX_{e_i+e_j})\cdot A\times A))\varphi(0,A),
\end{equation}
\begin{equation}\label{theta 11}
\Omega_{\psi}(u_{e_1-e_i}(t))\varphi(0,A)=\varphi(0,u_{e_1-e_i}(t).A),
\end{equation}
\begin{eqnarray}\label{theta 12}
\Omega_{\psi}(u_{e_i}(t))\varphi(0,A)&=&\psi(-\tr((tX_{e_i})\cdot A\times A))\varphi(0,u_{e_i}(t).A)\\
&=&\psi(-\tr((tX_{e_i})\cdot u_{e_i}(t).A\times u_{e_i}(t).A))\varphi(0,u_{e_i}(t).A).\nonumber
\end{eqnarray}

If we write $A\in \CJ$ as $A=\sum_i A_i X_{e_i}+\sum_{i,j}A_{ij}X_{e_i+e_j}$, let $\CJ_1$ be the subset containing those $A\in \CJ$ such that 
\begin{itemize}
\item $A_i=0$ for $i>1$;
\item $A_{ij}=0$ for all $i,j>1$ with $(i,j)\neq (3,4),(2,5)$;
\item $A_{25}=A_{34}\neq 0$.
\end{itemize}
Let $\CJ_{11}$ be the subset of $\CJ_1$ containing those $A\in \CJ_1$ with $A_1=A_{25}=A_{34}=\pm 1$ and let $\CJ_2$ (resp. $\CJ_{12}$) be the complement of $\CJ_1$ (resp. $\CJ_{11}$) in $\CJ$ (resp. $\CJ_1$). We also decompose $\CJ_{11}$ as $\CJ^+\cup \CJ^-$ where $\CJ^+$ (resp. $\CJ^-$) contains those elements with $A_1=A_{25}=A_{34}=1$ (resp. $A_1=A_{25}=A_{34}=-1$).

\begin{prop}The following four identities hold.
$$\int_{N(k)\back N(\BA)}\sum_{A\in \CJ_2(k)}\Omega_\psi(u)f\ast\varphi(0,A)\cdot \sum_{Y\in Y'(k)}\Omega_{\psi}'(u)\varphi'(Y)\ du=0,$$
$$\int_{N(k)\back N(\BA)}\sum_{A\in \CJ_{12}(k)}\Omega_\psi(u)f\ast\varphi(0,A)\cdot \sum_{Y\in Y'(k)}\Omega_{\psi}'(u)\varphi'(Y)\ du=0,$$
$$\int_{N(k)\back N(\BA)}\sum_{A\in \CJ^+(k)}\Omega_\psi(u)f\ast\varphi(0,A)\cdot \sum_{Y\in Y'(k)}\Omega_{\psi}'(u)\varphi'(Y) \ du=\int_{J^+(\BA)}\int_{Y'(\BA)}f\ast\varphi(0,A)\varphi'(Y)\psi(2Y\cdot A)\ dY\ dA,$$
$$\int_{N(k)\back N(\BA)}\sum_{A\in \CJ^-(k)}\Omega_\psi(u)f\ast\varphi(0,A)\cdot \sum_{Y\in Y'(k)}\Omega_{\psi}'(u)\varphi'(Y) \ du=\int_{J^-(\BA)}\int_{Y'(\BA)}f\ast\varphi(0,A)\varphi'(Y)\psi(2Y\cdot A)\ dY\ dA.$$
Here $Y\cdot A=(A_{12}-Y_5)(A_{15}-Y_2)+(A_{13}-Y_4)(A_{14}-Y_3)$.
\end{prop}

\begin{proof}
Let $N'$ be the subgroup of $N$ generated by the roots $e_i+e_j$ and let $N''$ be the subgroup of $N$ generated by the roots $e_1,e_i+e_j$. We have defined a character $\xi$ on $N''(\BA)$ to be $\xi(u)=\psi(\lambda(u))$ where $\lambda(u)$ is summation of the coordinates of the projection of $u$ into $U_{e_1},U_{e_2+e_5}$ and $U_{e_3+e_4}$. By \eqref{theta 10}-\eqref{theta 12}, we have (here we identify $N''$ with a subspace of $\CJ$ via the exponential map as we already identified $\CJ$ with the Lie algebra of $N$)
$$\Omega_\psi(u)f\ast\varphi(0,A)\Omega_{\psi}'(u)\varphi'(Y)=\xi(u)\xi(-\tr(u\cdot A\times A)f\ast\varphi(0,A)\varphi'(Y)$$
for all $A\in \CJ(\BA),Y\in Y'(\BA)$ and $u\in N'(\BA)$. If $A\in \CJ_2(k)$, it is clear that the character 
$$u\in N'(\BA) \mapsto \xi(u)\xi(-\tr(u)\cdot A\times A)$$ 
of $N'(\BA)$ is nontrivial, and this proves the first equation.

For the second equation, for $A\in \CJ_1$, by \eqref{theta 10}-\eqref{theta 12}, we have \footnote{Note that this would not be true if we do not assume $A\in \CJ_1$.}
$$\Omega_\psi(u)f\ast\varphi(0,A)\Omega_{\psi}'(u)\varphi'(Y)=\xi(u)\xi(-\tr(u)\cdot A\times A)f\ast\varphi(0,A)\varphi'(Y)$$
for all $Y\in Y'(\BA)$ and $u\in N''(\BA)$. It is clear that the character 
$$u\in N''(\BA)\rightarrow \xi(u)\xi(-\tr(u)\cdot A\times A)$$ 
of $N''(\BA)$ is trivial if and only if $A\in \CJ_{11}(k)$. This proves the second equation. 

The proof of the remaining two equations is similar, so we will only prove the third one. From our discussion above, we know that 
$$\int_{N(k)\back N(\BA)}\sum_{A\in \CJ^+(k)}\Omega_\psi(u)f\ast\varphi(0,A)\cdot \sum_{Y\in Y'(k)}\Omega_{\psi}'(u)\varphi'(Y)\ du$$
$$=\int_{N(k)N''(\BA)\back N(\BA)}\sum_{A\in \CJ^+(k)}\Omega_\psi(u)f\ast\varphi(0,A)\cdot \sum_{Y\in Y'(k)}\Omega_{\psi}'(u)\varphi'(Y)\ du.$$

By \eqref{theta 4}-\eqref{theta 6} and \eqref{theta 10}-\eqref{theta 12}, we have
$$\sum_{A\in \CJ^+(k)}f\ast\varphi(0,A)\cdot \sum_{Y\in Y'(k)}\varphi'(Y)=\sum_{u\in N''(k)\back N(k)} \Omega_\psi(u)f\ast\varphi(0,I)\cdot \Omega_{\psi}'(u)\varphi'(0)$$
and
$$\Omega_\psi(u)f\ast\varphi(0,I)\cdot \Omega_{\psi}'(u)\varphi'(0)=f\ast\varphi(0,u.I)\cdot \varphi'(u\star 0)\psi(2u\star 0\cdot u.I).$$
Here $I=X_{e_1}+X_{e_2+e_5}+X_{e_3+e_4}$ is the identity element of the Jordan algebra. Then, the above expression is equal to 
$$\int_{N''(\BA)\back N(\BA)} \Omega_\psi(u)f\ast\varphi(0,I)\cdot \Omega_{\psi}'(u)\varphi'(0)\ dA=\int_{\CJ^+(\BA)}\int_{Y'(\BA)}f\ast\varphi(0,A)\varphi'(Y)\psi(2Y\cdot A)\ dY\ dA$$
This proves the proposition.
\end{proof}

The above proposition implies that 
\begin{eqnarray*}
	&&\int_{N(k)\back N(\BA)}\sum_{A\in \CJ(k)}\Omega_\psi(u)f\ast\varphi(0,A)\cdot \sum_{Y\in Y'(k)}\Omega_{\psi}'(u)\varphi'(Y)\ du\\
&=&\int_{\CJ^+(\BA)}\int_{Y'(\BA)}f\ast\varphi(0,A)\varphi'(Y)\psi(2Y\cdot A)\ dY\ dA\\
&&+\int_{\CJ^-(\BA)}\int_{Y'(\BA)}f\ast\varphi(0,A)\varphi'(Y)\psi(2Y\cdot A)\ dY\ dA\\
&=&I^+(f\ast \varphi,\varphi')+I^-(f\ast\varphi,\varphi')=\Pi_{v\in |k|}I_{v}^{+}(f_v\ast \varphi_v,\varphi_v')+\Pi_{v\in |k|}I_{v}^{-}(f_v\ast \varphi_v,\varphi_v')
\end{eqnarray*}
where
$$I_{v}^{\varepsilon}(f_v\ast \varphi_v,\varphi_v')=\int_{\CJ^\varepsilon(k_v)}\int_{Y'(k_v)}f_v\ast\varphi_v(0,A)\varphi_v'(Y)\psi(2Y\cdot A)\ dY\ dA,\;\varepsilon\in \{\pm\}.$$

The same argument also applies to the case when $\hat{G}=S(\GSpin_7\times \GL_2)$. The only difference is that in this case $\CJ$ is $Mat_{1\times 1}\times Asym_{4\times 4}$ and we can identify it with the Lie algebra of the Siegel unipotent subgroup, i.e. we identify $Mat_{1\times 1}$ with the unipotent of the Borel subgroup of $\GL_2$, and we identify $Asym_{4\times 4}$ with the root spaces $X_{e_i},X_{e_i+e_j}$ for $1\leq i\neq j\leq 3$. Also in this case $Y'$ is spanned by the root spaces of $e_1-e_2$ and $e_1-e_3$. To summarize, we have proved the following proposition.

\begin{prop}
For $f=\otimes_{v\in |k|}f_v,\;\varphi=\otimes_{v\in |k|}\varphi_v$ and $\varphi'=\otimes_{v\in |k|}\varphi_v'$, we have
$$I(f,\varphi,\varphi')=\Pi_{v\in |k|}I_{v}^{+}(f_v\ast \varphi_v,\varphi_v')+\Pi_{v\in |k|}I_{v}^{-}(f_v\ast \varphi_v,\varphi_v')+\sum_{a\in k^\times} \Pi_{v\in |k|} I_v(a,f_v\ast \varphi_v,\varphi_v')$$
where
$$I_v(a,f_v\ast\varphi_v,\varphi_v')=\int_{\CJ(k_v)}\int_{Y'(k_v)}\psi_v(\frac{\tr(A)-N_\CJ(A)-\pair{Y,A}}{a}-\frac{\pair{A,A}}{2a^2}) f_v\ast\varphi_v(a,A)\varphi_v'(Y)\ dY\ dA$$
and
$$I_{v}^{\varepsilon}(f_v\ast \varphi_v,\varphi_v')=\int_{\CJ^\varepsilon(k_v)}\int_{Y'(k_v)}f_v\ast\varphi_v(0,A)\varphi_v'(Y)\psi(2Y\cdot A)\ dY\ dA,\;\varepsilon\in \{\pm\}.$$
\end{prop}

Now, we can formulate the fundamental lemma and smooth transfer.

\begin{conj}\label{conj case 2}
With the notation above, let $n=\dim(\CJ)$. The following holds.
\begin{enumerate}
\item (Fundamental lemma) Assume that $v$ is non-archimedean, the residue characteristic is not equal to 2, and $\psi_v$ is unramified over $v$. Let $\CO_v$ be the ring of integers of $k_v$ and let $f_{0,v}'$ (resp. $f_{0,v},\;\varphi_{0,v},\;\varphi_{0,v}'$) be the characteristic function of $G'(\CO_{v})$ (resp. $\hat{G}(\CO_v)Z_{\hat{G}}(k_v)$, $Y(\CO_v)$, $Y'(\CO_v)$). Then (note that $f_{0,v}\ast \varphi_{0,v}=\varphi_{0,v}$)
$$I_{v}^{+}(\varphi_{0,v},\varphi_{0,v}')=J_{v}^{+}(f_{0,v}'),I_{v}^{-}(\varphi_{0,v},\varphi_{0,v}')=J_{v}^{-}(f_{0,v}'),$$
$$I_v(a,\varphi_{0,v},\varphi_{0,v}')=|4|\cdot |a|^{(n+1)/2}J_v(a,f_{0,v}'),\;\forall a\in k_{v}^{\times}.$$
\item (Smooth transfer) For any $f_v'\in \CS(G'(k_v))$ (resp. $f_v\in \CS(\hat{G}(k_v)/Z_{\hat{G}}(k_v))$, $\varphi_v\in \CS(Y(k_v))$ and $\varphi_v'\in \CS(Y'(k_v))$), there exists $f_v\in \CS(\hat{G}(k_v)/Z_{\hat{G}}(k_v))$, $\varphi_v\in \CS(Y(k_v))$ and $\varphi_v'\in \CS(Y'(k_v))$ (resp. $f_v'\in \CS(G'(k_v))$) such that
$$I_{v}^{+}(f_v\ast \varphi_v,\varphi_v')=J_{v}^{+}(f_v'),\;I_{v}^{-}(f_v\ast \varphi_v,\varphi_v')=J_{v}^{-}(f_v'),\;I_v(a,f_v\ast\varphi_v,\varphi_v')=|4|\cdot |a|^{(n+1)/2}J_v(a,f_v'),\;\forall a\in k_{v}^{\times}.$$
\end{enumerate}
\end{conj}

\subsection{The proof of the fundamental lemma and smooth transfer}

In this subsection we will prove the fundamental lemma and smooth transfer in the p-adic case. To simplify the notation, we will use $F$ to replace the local field $k_v$ and eliminate all the subscript $v$. We also assume that $F$ is non-archimedean. By Proposition \ref{orbital integral SL(2)}, we only need to prove the following theorem.

\begin{thm}\label{orbital integral Model 5-6}
\begin{enumerate}
\item For $F\in C_{c}^{\infty}(F^\times)$, there exists $\varphi\in C_{c}^{\infty}(Y(F))$ and $\varphi'\in C_{c}^{\infty}(Y'(F))$ such that
$$I(a,\varphi,\varphi')=F(a),\;\forall a\in F^\times.$$
\item If $\varphi$ (resp. $\varphi'$) is the characteristic function of $Y(\CO_F)$ (resp. $Y'(\CO_F)$), $\psi$ is unramified, and the residue characteristic of $F$ is not 2, then $I^+(\varphi,\varphi')=I^-(\varphi,\varphi')=1$ and
$$I(a,\varphi,\varphi')=0\;\text{if}\;|a|>1;\; I(a,\varphi,\varphi')=1\;\text{if}\;|a|=1;\; I(a,\varphi,\varphi')=|a|^{(n-1)/2}\int_{\CO_{F}^{\times}}\psi(\frac{x+x^{-1}}{a})\ dx\;\text{if}\; |a|<1.$$
\item For $\varphi\in C_{c}^{\infty}(Y(F))$ and $\varphi'\in C_{c}^{\infty}(Y'(F))$, the function
$$a\in F^\times\mapsto |a|^{(1-n)/2}I(a,\varphi,\varphi')$$
belongs to the space $C_{OI}^{\infty}(F^\times,c_+,c_-)$ for $c_+=|4|\cdot I^+(\varphi,\varphi')$ and $c_-=|4|\cdot I^-(\varphi,\varphi')$.
\end{enumerate}
\end{thm}

We will only consider the case when $\hat{G}=\Spin_{11}$, the case when $\hat{G}=S(\GSpin_7\times \GL_2)$ follows from a similar and easier argument. In this case $n=\dim(\CJ)=15$. The first part is trivial. For the second part, the identity $I^+(\varphi,\varphi')=I^-(\varphi,\varphi')=1$ is trivial. Also it is clear that $I(a,\varphi,\varphi')=0$ if $|a|>1$ and $I(a,\varphi,\varphi')=0$ if $|a|=1$. So, it remains to show that 
$$I(a,\varphi,\varphi')=|a|^{(n-1)/2}\int_{\CO_{F}^{\times}}\psi(\frac{x+x^{-1}}{a})\ dx$$
if $|a|\leq 1$. With the same notation as in \cite{MR}, we can write $A$ as $\Pi_{\Gamma_0\in \Pi_0'} x_\gamma$. We can also decompose $\Pi_0'$ as $\{\alpha+\beta\}\cup \Pi'\cup \{\gamma_1,\cdots,\gamma_r\}\cup \{\delta_1,\cdots,\delta_r\}$. Under our notation above, we have 
$$x_{\alpha+\beta}=A_1,\;x_{\gamma_1}=A_{25},x_{\gamma_2}=A_{24},x_{\gamma_3}=A_{23},x_{\delta_1}=A_{34},x_{\delta_2}=A_{35},x_{\delta_3}=A_{45}$$
and $\{x_{\gamma}|\;\gamma\in \Pi'\}$ is just $\{A_{i},A_{1i}|\;2\leq i\leq 5\}$.

Assume that the residue characteristic of $F$ is not 2. Then we have 
\begin{eqnarray*}
I(a,\varphi,\varphi')&=&\int_{\CJ(\CO_F)}\int_{Y'(\CO_F)}\psi(\frac{\tr(A)-N_\CJ(A)-\pair{Y,A}}{a}-\frac{\pair{A,A}}{2a^2}) \ dY\ dA\\
&=&\int_{\CJ(\CO_F)_a}\psi(\frac{\tr(A)-N_\CJ(A)}{a}-\frac{\pair{A,A}}{2a^2}) \ dA\\
&=&\int_{\CJ(\CO_F)_a}\psi(\frac{\tr(A)-N_\CJ(A)}{a}) \ dY\ dA
\end{eqnarray*}
where $\CJ(\CO_F)_a=\{A\in \CO_F|\; A_i\in a\CO_F,\;\forall 2\leq i\leq 5\}$. Then we can use the same argument as in the proof of Theorem 2 of \cite{MR} to finish the proof. The only difference is that in the last paragraph of Section 4 of \cite{MR}, instead of using the identity in Lemma 4 of loc. cit., we use the following identity whose proof is trivial. 

\begin{equation}
\int_{\varpi^m\CO_F}\int_{\CO_F}\psi(bxy)\ dx\ dy=|b|^{-1},\;\forall |b|>q^m,m>0.
\end{equation}

It remains to prove the last part. We may assume that $\psi$ is unramified. We only need to consider the case when $\varphi$ (resp. $\varphi'$) is the product of the characteristic function of $\varpi^m\CO_F$ and the characteristic function of $A_0+\varpi^m\CJ(\CO_F)$ (resp. the characteristic function of $Y_0+\varpi^m Y'(\CO_F)$). We first show that 

\begin{enumerate}
    \item[(4)] if the support $A_0+\varpi^m\CJ(\CO_F)$ has no intersection with $\CJ^+(F)\cup \CJ^-(F)$, then the function $a\mapsto I(a,\varphi,\varphi')$ belongs to $C_{c}^{\infty}(F^\times)$.
\end{enumerate}

 In this case, by choosing $m$ large enough we may assume that $A\times A\notin \CJ^+(F)$ for all $A\in A_0+\varpi^m\CJ(\CO_F)$. As in the previous case, for each $\gamma\in \Pi_0'$, we can write $\tr(A)-N_\CJ(A)$ as $x_\gamma P_\gamma(A)+Q_\gamma(A)$ where $P_\gamma(A)$ and $Q_\gamma(A)$ are polynomials in $\{x_{\gamma'}|\;\gamma'\neq \gamma\}$. Since $A\times A\notin \CJ^+(F)$ for all $A\in A_0+\varpi^m\CJ(\CO_F)$, there exists $\gamma\in \Pi_0'$ and $C>0$ such that $|P_\gamma(A)|>C$ for all $A\in A_0+\varpi^m\CJ(\CO_F)$ and $x_\gamma\neq A_i$ for $2\leq i\leq 5$. This implies that for $a$ small enough, the orbital integral $I(a,\varphi)$ is equal to 0 (note that $\pair{Y,A}$ and $\pair{A,A}$ only depends on $Y_i$ and $A_i$ for $2\leq i\leq 5$). This proves (4).

Hence, it is enough to consider the case when $A_0\in \CJ^+(F)$ or $A_0\in \CJ^-(F)$. The arguments for both cases are the same, so we may just assume that $A_0\in \CJ^+(F)$. In this case, we may write $Y_0$ and $A_0$ as
$$Y_0=\sum_{2\leq i\leq 5}Y_{i,0}X_{e_1-e_i},\;A_0=1+\sum_{2\leq i\leq 5}A_{1i,0}X_{e_1+e_i}.$$
Also by (4), we can replace the test function $\varpi$ by its product with the characteristic function $\textbf{1}_{1+\varpi^{2m}\CO_F}(A_1)$ of $A_1$.

By the same argument as in the proof of Theorem 2 of \cite{MR} (except that we replace Lemma 4 of loc. cit. by \eqref{basic equation 1}), we have the following identity, which is an analog of equation (41) of loc. cit. \footnote{In our current case, we not only have the character $\psi(\frac{\tr(A)-N(A)}{a})$ as in \cite{MR}, but we also have the extra part $\psi(\frac{-\pair{Y,A}}{a}-\frac{\pair{A,A}}{2a^2})$, this is why the argument in loc. cit. can only help us calculate the integrals over $A_{ij}$ for $2\leq i,j\leq 5$.}
\begin{eqnarray*}
	I(a,\varphi,\varphi')&=&|a|^{3}\int_{(\varpi^m\CO_F)^4}\int_{\Pi_{2\leq i\leq 5}A_{1i,0}+\varpi^m\CO_F}\int_{\Pi_{2\leq i\leq 5}Y_{i,0}+\varpi^m\CO_F}\int_{1+\varpi^{2m}\CO_F}\\ 
&&\psi(\frac{x+\frac{1}{x}}{a}+\frac{\sum_{2\leq i\leq 5}Y_iA_i}{a}-\frac{\sum_{2\leq i\leq 5}A_{7-i}A_{1i}}{ax}-\frac{A_2A_5+A_3A_4}{2a^2}) \   dx\Pi_{2\leq i\leq 5}\ dY_i\ dA_{1i}\ dA_i.
\end{eqnarray*}

We can first compute the integral over $Y_i$. We have
$$\int_{\Pi_{2\leq i\leq 5}Y_{i,0}+\varpi^m\CO_F} \psi(\frac{\sum_{2\leq i\leq 5}Y_iA_i}{a})\Pi_{2\leq i\leq 5}\ dY_i=q^{-4m}\cdot \textbf{1}_{(a\varpi^{-m}\CO_F)^4}(A_2,A_3,A_4,A_5)\psi(\frac{\sum_{2\leq i\leq 5} Y_{i,0}A_i}{a})$$
once we let $a$ be small enough with respect to $\varpi^m$. This implies that
\begin{eqnarray*}
	I(a,\varphi,\varphi')&=&|a|^{3}q^{-4m}\cdot\int_{(a\varpi^{-m}\CO_F)^4} \int_{\Pi_{2\leq i\leq 5}A_{1i,0}+\varpi^m\CO_F}\int_{1+\varpi^{2m}\CO_F}\\
&&\psi(\frac{x+\frac{1}{x}}{a}+\frac{\sum_{2\leq i\leq 5} Y_{i,0}A_i}{a}-\frac{\sum_{2\leq i\leq 5}A_{7-i}A_{1i}}{ax}-\frac{A_2A_5+A_3A_4}{2a^2})  \  dx\Pi_{2\leq i\leq 5}\ dA_{1i}\ dA_i\\
&=&|a|^{7}q^{-4m}\cdot\int_{(\varpi^{-m}\CO_F)^4} \int_{\Pi_{2\leq i\leq 5}A_{1i,0}+\varpi^m\CO_F}\int_{1+\varpi^{2m}\CO_F}\\
&&\psi(\frac{x+\frac{1}{x}}{a}+\sum_{2\leq i\leq 5} Y_{i,0}A_i-\frac{\sum_{2\leq i\leq 5}A_{7-i}A_{1i}}{x}-\frac{A_2A_5+A_3A_4}{2})  \  dx\Pi_{2\leq i\leq 5}\ dA_{1i}\ dA_i.
\end{eqnarray*}
Then we can compute the integral over $A_{1i}$. We have (note that $x\in 1+\varpi^{2m}\CO_F$ and $A_i\in \varpi^{-m}\CO_F$)
$$\int_{\Pi_{2\leq i\leq 5}A_{1i,0}+\varpi^m\CO_F} \psi(-\frac{\sum_{2\leq i\leq 5}A_{7-i}A_{1i}}{x}) \Pi_{2\leq i\leq 5}\ dA_{1i}=q^{-4m}\cdot \psi(-\frac{\sum_{2\leq i\leq 5}A_{7-i}A_{1i,0}}{x})$$
once we choose $m$ large enough with respect to $A_{1i,0}$. This implies that 
\begin{eqnarray*}
	I(a,\varphi,\varphi')&=&|a|^{7}q^{-8m}\cdot\int_{(\varpi^{-m}\CO_F)^4} \int_{1+\varpi^{2m}\CO_F}\\
&&\psi(\frac{x+\frac{1}{x}}{a}+\sum_{2\leq i\leq 5} Y_{i,0}A_i-\frac{\sum_{2\leq i\leq 5}A_{7-i}A_{1i,0}}{x}-\frac{A_2A_5+A_3A_4}{2})  \  dx\Pi_{2\leq i\leq 5}\ dA_i.
\end{eqnarray*}

Once we choose $m$ large enough (with respect to $A_{1i,0}$), we know that 
$$\psi(\frac{\sum_{2\leq i\leq 5}A_{7-i}A_{1i,0}}{x})=\psi(\sum_{2\leq i\leq 5}A_{7-i}A_{1i,0})$$
for all $A_i\in \varpi^{-m}\CO_F$. Hence we have
\begin{eqnarray*}
I(a,\varphi,\varphi')&=&|a|^{7}q^{-8m}\cdot \int_{1+\varpi^{2m}\CO_F}\psi(\frac{x+\frac{1}{x}}{a})\ dx\\
&&\times \int_{(\varpi^{-m}\CO_F)^4} 
\psi(\sum_{2\leq i\leq 5} (Y_{i,0}A_i-A_{7-i}A_{1i,0})-\frac{A_2A_5+A_3A_4}{2})  \Pi_{2\leq i\leq 5}dA_i\\
&=&|a|^{7}q^{-8m}\cdot \psi(2(Y_{2,0}-A_{15,0})(Y_{5,0}-A_{12,0})+2(Y_{3,0}-A_{14,0})(Y_{4,0}-A_{13,0}))\\
&&\times \int_{1+\varpi^{2m}\CO_F}\psi(\frac{x+\frac{1}{x}}{a})\ dx\cdot  \int_{(\varpi^{-m}\CO_F)^4} 
\psi(-\frac{A_2A_5+A_3A_4}{2})    \Pi_{2\leq i\leq 5}\ dA_i\\
&=&|a|^{7}q^{-8m}|4|\cdot \psi(2(Y_{2,0}-A_{15,0})(Y_{5,0}-A_{12,0})+2(Y_{3,0}-A_{14,0})(Y_{4,0}-A_{13,0}))\\
&&\times\int_{1+\varpi^{2m}\CO_F}\psi(\frac{x+\frac{1}{x}}{a})\ dx.
\end{eqnarray*}
Hence, it remains to show that 
$$I^+(\varphi,\varphi')=q^{-8m}\cdot \psi(2(Y_{2,0}-A_{15,0})(Y_{5,0}-A_{12,0})+2(Y_{3,0}-A_{14,0})(Y_{4,0}-A_{13,0})).$$

From the definition, we know that 
\begin{eqnarray*}I^+(\varphi,\varphi')&=&\int_{\Pi_{2\leq i\leq 5}A_{1i,0}+\varpi^m\CO_F}\int_{\Pi_{2\leq i\leq 5}Y_{i,0}+\varpi^m\CO_F}\\
&&\psi(2(A_{12}-Y_5)(A_{15}-Y_2)+2(A_{13}-Y_4)(A_{14}-Y_3))\Pi_{2\leq i\leq 5}\ dY_i\ dA_{1i}.\end{eqnarray*}
Once we choose $m$ large enough (with respect to $A_{1i,0}, Y_{i,0}$), we know that the integrand is just a constant equal to 
$$\psi(2(Y_{2,0}-A_{15,0})(Y_{5,0}-A_{12,0})+2(Y_{3,0}-A_{14,0})(Y_{4,0}-A_{13,0})).$$
This finishes the proof of the theorem.

\section{More examples of Conjecture \ref{conj Whittaker general}} \label{sec orbitexample}

In this section, we discuss two more examples of Conjecture \ref{conj Whittaker general}. In the first example the set $\{\tau_j\oplus \tau_{j}^{\vee}\}$ in the decomposition \ref{decomposition of odd 1} is not empty; in the second example  $\hat{\CO}_\iota$ is not special.

\subsection{An example for $\GL_n$}
In this subsection we consider the case when $G=\hat{G}=\GL_n$ ($n>2$) and $\hat{\CO}_\iota$ is the nilpotent of $\hat{G}$ with partition $(n-2,1,1)$, i.e. it is a principal orbit in the Levi subgroup $\GL_{n-2}\times \GL_1\times \GL_1$. In this case its dual nilpotent orbit $\CO_\iota$ has partition $(3,1,1,1,\cdots,1)$ and the degenerate Whittaker period $\CP_\iota$ is given by 
$$\CP_\iota(\phi)=\int_{U(k)\back U(\BA)}\phi(u)\mu(u)du$$
where $U$ and $\mu$ are defined in (2.31) of \cite{MR4}. When $n=3$, this is just the Whittaker period. For the rest of this subsection, we assume that $n>3$.

In this case, the centralizer of $Im(\hat{\iota})$ in $\hat{G}$ is $\hat{H}'=\GL_2\times \GL_1$ and the adjoint action of $\hat{G}\times \SL_2$ on $\hat{\Fg}$ can be decomposed as (for simplicity we will ignore the $\GL_1$ part of $\hat{H}$)
$$\hat{\Fg}=(Ad\otimes \wedge^2)\oplus (std\otimes Sym^{n-3})\oplus (std^\vee\otimes Sym^{n-3})\oplus (\wedge^2\otimes \tau)$$
where $\tau$ is some representation of $\SL_2$. Since the L-value associated to $\wedge^2\otimes \tau$ is just some zeta factors we will ignore it. In this case, the decomposition \eqref{decomposition of odd 1} becomes
$$\oplus_{k\in \hat{I}_{odd}}\hat{\rho_k}=(std\oplus std^\vee)$$
if $n$ is even and 
$$\oplus_{k\in \hat{I}_{odd}}\hat{\rho_k}=0$$
if $n$ is odd. In particular, this is an example for which the set $\{\tau_j\oplus \tau_{j}^{\vee}\}$ in the decomposition \ref{decomposition of odd 1} is not empty. In both cases we have
$\rho_{\hat{\iota}}=0$.

Let $H=\GL_{n-1}=\{\begin{pmatrix}h&0\\0&1 \end{pmatrix}\}$ be a closed subgroup of $G$. By the unramified computation in \cite{S1} and the computation of dual group in \cite{KS1}, the quadruples $(G,H,0,1)$ and $(\hat{G},\hat{H}',0,\hat{\iota})$ are dual to each other. Since $\rho_{\hat{\iota}}=0$, we know that the quadruple $(G,H,0,1)$ satisfies Assumption \ref{assumption 1}. In this case, the relative trace formula comparison in Conjecture \ref{conj RTF} recovers the one in (2.34) of \cite{MR4}. In principle the ideas of \cite{MR4} can establish this comparison (and we will do so in a future paper). This comparison serves as strong evidence of Conjecture \ref{conj Whittaker general} for the period $\CP_\iota$.

\subsection{An example of Conjecture \ref{conj Whittaker general} for non-special nilpotent orbit}\label{sec non special}
In this section, we will provide an example of Conjecture \ref{conj Whittaker general} in the case when $\hat{\CO}_\iota$ is not special. The case we are considering is when $G=\Sp_{4n}$ and $\hat{G}=\SO_{4n+1}$. Let ${e_i-e_j,e_i+e_j,e_k}$ (resp. ${e_i-e_j,e_i+e_j,2e_k}$) be the roots of $\SO_{4n+1}$ (resp. $\Sp_{2n}$) with $1\leq i,j,k\leq 2n$ and $i\neq j$. Let $P=MN$ (resp. $\hat{P}=\hat{M}\hat{N}$) be the standard parabolic subgroup of $\Sp_{4n}$ (resp. $\SO_{4n+1}$) associated to the simple roots $e_{2i-1}-e_{2i}$ for $1\leq i\leq n$. In particular we have $M \simeq \hat{M}\simeq (\GL_2)^n$.

Let  $\hat{\CO}_\iota$ be the principal nilpotent orbit of $\hat{\Fm}$, and we also view it as a nilpotent orbit of $\hat{\Fg}$. It is clear that it is not a special nilpotent orbit of $\Fg$. The nilpotent orbit $\hat{\CO}_\iota$ induces a homomorphism $\hat{\iota}:\SL_2\rightarrow \hat{G}$. Let $\hat{H}$ be the neutral component of the centralizer of $Im(\hat{\iota})$ in $\hat{G}$. It is easy to see that in this case we have $\hat{H}=\Sp_{2n}$. We get a homomorphism $\hat{H}'\times \SL_2\rightarrow \hat{G}$ which will still be denoted by $\hat{\iota}$. This map induces an adjoint action of $\hat{H}\times \SL_2$ on $\hat{\Fg}$, and it is easy to see that we can decompose it as
$$Ad\otimes Sym^0\oplus std\otimes Sym^1\oplus \wedge^2\otimes Sym^2$$
where $Ad$ is the adjoint representation of $\Sp_{2n}$, $std$ is the standard representation of $\Sp_{2n}$ and $\wedge^2$ is the exterior square representation. In particular, in this case, we have $\rho_{\hat{\iota}}=std$.

Let $\CO_\iota$ be the dual of $\hat{\CO}_\iota$. It is easy to see that $\CO_\iota$ is the Richardson nilpotent orbit associated to parabolic subgroup $P$ of $G$. Moreover, it induces a character $\xi$ of $N(k)\back N(\BA)$ and we can define the degenerate Whittaker period to be 
$$\CP_\iota(\phi)=\int_{N(k)\back N(\BA)}\phi(n)\xi(n)dn.$$
In this case, Conjecture \ref{conj Whittaker general} becomes the following conjecture.

\begin{conj}\label{conj Whittaker Sp(4n)}
Let $\pi$ be an irreducible discrete automorphic representation of $\hat{G}(\BA)$ and let $\nu:\pi\rightarrow L^2(\hat{G}(k)\back \hat{G}(\BA))_{\pi}$ be an embedding. Assume that the Arthur parameter of $\pi$ factors through $\hat{\iota}$ (i.e. $\pi$ is a lifting of a global tempered Arthur packet $\Pi$ of $H(\BA)=\SO_{2n+1}(\BA)$). Then one can choose the embedding $\nu$ so that 
$$\frac{|\CP_\iota(\phi)|^2}{\pair{\phi,\phi}}``=" \frac{L(1/2,\Pi)}{L(1,\Pi,Ad)L(3/2,\Pi)L(2,\Pi,\wedge^2)},\;\phi\in \nu(\pi).$$
\end{conj}

Next, we define another unipotent period integral on $\Sp_{4n}$. Let $B=TU$ be the standard upper triangular Borel subgroup of $\Sp_{4n}$, and we define a character $\xi'$ on $U(k)\back U(\BA)$ to be (note that this is not a generic character)
$$\xi'(u)=\psi(u_{12}+u_{23}+\cdots+u_{2n-1\;2n}),\;u=(u_{ij})_{1\leq i,j\leq 4n}.$$
Then we define the period integral 
$$\CP_{U,\xi'}(\phi)=\int_{U(k)\back U(\BA)}\phi(u)\xi'(u)du.$$

Lastly, let $G'=\Sp_{2n}\times \Sp_{2n}$ be a closed subgroup of $G$ and we define the period integral
$$\CP_{G'}(\phi)=\int_{G'(k)\back G'(\BA)}\phi(h)dh.$$
By the unramified computation in \cite{S1} and the computation of dual group in \cite{KS1}, we know that the quadruples
$$(G,G',0,1),\;(\hat{G},\hat{H},std,\hat{\iota})$$
are dual to each other under the BZSV duality. Hence, for the period integral $\CP_{G'}$, Conjecture \ref{BSV conj} becomes the following conjecture.

\begin{conj}\label{conj Sp(4n),Sp(2n)}
Let $\pi$ be an irreducible discrete automorphic representation of $G(\BA)$ and let $\nu:\pi\rightarrow L^2(G(k)\back G(\BA))_{\pi}$ be an embedding. Then the period integral $\CP_{G'}(\phi)$ is nonzero for some $\phi\in Im(\nu)$ only if the Arthur parameter of $\pi$ factors through $\hat{\iota}:\hat{H}(\BC)\times \SL_2(\BC)\rightarrow \hat{G}(\BC)$. If this is the case, $\pi$ is a lifting of a global tempered Arthur packet $\Pi$ of $H(\BA)=\SO_{2n+1}(\BA)$. Then we can choose the embedding $\nu$ so that 
$$\frac{|\CP_{G'}(\phi)|^2}{\pair{\phi,\phi}}``="  \frac{L(1/2,\Pi)L(3/2,\Pi)L(2,\Pi,\wedge^2)}{L(1,\Pi,Ad)},\;\phi\in Im(\nu).$$
\end{conj}

Next, we introduce three relative trace formulas. Let $f$ (resp. $f'$) be a Schwartz function on $G(\BA)$ (resp. $\GL_{2n}(\BA)$) and $K_f(x,y)$ (resp. $K_{f'}(x,y)$) be the kernel function. On the $\GL_{2n}$ side, we define
$$I(f')=\int_{\GL_n\times \GL_n(k)\back \GL_n\times \GL_n(\BA)/Z_{\GL_{2n}}(\BA)}\int_{N_{\GL_{2n}}(k)\back N_{\GL_{2n}}(\BA)} K_{f'}(x,y)\xi_{2n}^{-1}(y)\ dy\ dx$$
where $N_{\GL_{2n}}$ is a maximal unipotent subgroup of $\GL_{2n}$ and $\xi_{2n}$ is a generic character of it. Here the $\GL_n\times \GL_n$-period is known as the linear period. It has been studied by Friedberg and Jacquet in \cite{FJ1}. By the result in loc. cit., the discrete part of the spectral expansion of $I(f')$ is of the form
\begin{equation}\label{spectral expansion 3}
I_{disc}(f')``="\sum_{\Pi} \frac{L(1/2,\Pi)}{L(1,\Pi,Ad)}.
\end{equation}
where $\Pi$ runs over all the global discrete generic Arthur packet of $\SO_{2n+1}(\BA)$.

On the other side, we define
$$J(f)=\int_{G'(k)\back G'(\BA)}\int_{N(k)\back N(\BA)}K_f(x,y)\xi(y)^{-1}dydx,$$
$$J'(f)=\int_{G'(k)\back G'(\BA)}\int_{U(k)\back U(\BA)}K_f(x,y)\xi'(y)^{-1}dydx.$$

If we assume Conjecture \ref{conj Whittaker Sp(4n)} and \ref{conj Sp(4n),Sp(2n)} to be true, then the discrete part of the spectral expansion of $J(f)$ should also be of the form \eqref{spectral expansion 3}. As a result, we should expect the following conjecture. 

\begin{conj}\label{conj Sp(2n) RTF}
There is a comparison between the relative trace formulas $I(f')$ and $J(f)$.
\end{conj}

Conversely, if we have such a comparison, we should expect Conjecture \ref{conj Whittaker Sp(4n)} to be true. 

In Theorem 2.1 of \cite{MR3}, they proved a comparison between the relative trace formulas $I(f')$ and $J'(f)$. Together with the following proposition, we should expect the relative trace formula comparison in Conjecture \ref{conj Sp(2n) RTF} to hold and Conjecture \ref{conj Whittaker Sp(4n)} to be true. This is an example of Conjecture \ref{conj Whittaker general} in the case when the nilpotent orbit is not special.

\begin{prop}\label{prop model transit}
Let $\pi$ be an irreducible discrete automorphic representation of $\hat{G}(\BA)$ and let $\nu:\pi\rightarrow L^2(\hat{G}(k)\back \hat{G}(\BA))_{\pi}$ be an embedding. Assume that the period integral $\CP_{G'}$ is non-vanishing on $Im(\nu)$. Then 
$$\CP_\iota(\phi)``="\CP_{U,\xi'}(\phi),\;\phi\in Im(\nu).$$
\end{prop}
\begin{proof}
	This follows essentially from \cite[Theorem~16]{GRS2} and \cite[\S5 Theorem 1 and Lemma 2]{GRS3}. The result of \cite[\S5 Theorem 1]{GRS3} is an equation for $\CP_\iota(\CE)$ (the equation (5.18) in \cite{GRS3}), where $\CE$ is a residue of an Eisenstein series. However the proof applies to any automorphic form $\phi$ in an irreducible representation with nontrivial $\Sp_{2n}\times\Sp_{2n}$ linear form, as by \cite[Theorem~16]{GRS2} such a $\phi$ satisfies the key vanishing result \cite[(5.2)]{GRS3}. Thus following their argument we get (\cite[(5.16)]{GRS3})
	$$\CP_\iota(\phi)=\int_{X_0(\A)}\CP_{U,\xi'}(\phi(\cdot \bar{\ell}(x)\nu_0))\ dx.$$
	Here $\nu_0$ is a Weyl element and $\bar{\ell}(X_0)$ is a subgroup of $\Sp_{4n}$ defined in \cite[(4.12)]{GRS3}. In particular
	$\CP_\iota(\phi)``="\int_{X_0(\A)}\CP_{U,\xi'}(\phi(\cdot \bar{\ell}(x))\ dx$.
	
In \cite[Lemma 2]{GRS3} they proved the non-vanishing of the above integral. To do so they defined a sequence of subgroups $\{e\}=K^{2n}\subset K^{2n-1}\subset\cdots\subset K^1=\bar{\ell}(X_0)$ in $\Sp_{4n}$, and integrals $a_i(f)=\int_{K^i(\A)}f(x)\ dx$. Thus $\CP_\iota(\phi)``="a_1(f)$ when $f(g)=\CP_{U,\xi'}(\phi(\cdot g))$. Next we can define a sequence of abelian subgroups $R_i\subset \Sp_{4n}$ ($1\leq i\leq 2n$) as in \cite[(5.31)]{GRS3}. The key statement in the proof of \cite[Lemma 2]{GRS3} is that if one expands $f$ as
	$$f=\sum_\alpha \int_{R_i(\A)}\varphi_\alpha(r) f_\alpha(\cdot r)\ dr,$$
	then $a_i(f)=a_{i-1}(f')$ where
	$$f'=\sum_\alpha \int_{R_i(\A)}\hat{\varphi}_\alpha(r) f_\alpha(\cdot r)\ dr,$$
	here $\hat{\varphi}_\alpha$ is the Fourier transform of $\varphi_\alpha$. In particular we have $a_i(f)``="a_{i-1}(f)$, and $f(1)=a_{2n}(f)``="a_{1}(f)$. Thus $\CP_\iota(\phi)``="\CP_{U,\xi'}(\phi)$.
	\end{proof}

\end{document}